\documentclass{amsart}
\usepackage[utf8]{inputenc}

\usepackage{amsmath}
\usepackage{amssymb}
\usepackage{stmaryrd}
\usepackage{tikz-cd}
\usepackage[shortlabels]{enumitem}
\usepackage{bbm}
\usepackage{bbold}
\usepackage{todonotes}
\usepackage{quiver}
\usepackage{url}
\usepackage{hyperref}
\usepackage{stmaryrd}
\usepackage{accents}

\newcommand{\axiom}[1]{\mathsf{#1}}
\newcommand{\ZFC}{\axiom{ZFC}}

\newcommand{\ZF}{\axiom{ZF}}

\newcommand{\DC}{\axiom{DC}}

\newcommand{\AD}{\axiom{AD}}

\newcommand{\sI}{\mathcal{I}}

\DeclareMathOperator{\add}{add}
\DeclareMathOperator{\non}{non}

\DeclareMathOperator{\proj}{proj}

\DeclareMathOperator{\dom}{dom}

\DeclareMathOperator{\trcl}{trcl}

\DeclareMathOperator{\Coll}{Coll}

\theoremstyle{plain}
\newtheorem{thm}{Theorem}[section]
\newtheorem{lemma}[thm]{Lemma}

\newtheorem{cor}[thm]{Corollary}
\newtheorem{claim}{Claim}[section]
\newtheorem{fact}[thm]{Fact}
\newtheorem*{remark}{Remark}
\theoremstyle{definition}
\newtheorem{definition}[thm]{Definition}

\newtheorem{exmp}[thm]{Example}
\newtheorem{quest}[thm]{Question}

\date{\today}
\title{Universal domination and idealized forcing}
\author{Jonathan Schilhan}
\address{University of Vienna\\
Institute of Mathematics\\
Kurt Gödel Research Center\\
Kolingasse 14-16\\
1090 Vienna\\
Austria}
\email{jonathan.schilhan@univie.ac.at}
\thanks{This research was funded in whole or in part by the Austrian Science Fund (FWF) [10.55776/ESP5711024]. For open access purposes, the authors have applied a CC BY public copyright license to any author-accepted manuscript version arising from this submission.}

\begin{document}

\begin{abstract}
    We introduce the \emph{universality property} of a definable $\sigma$-ideal on a Polish space, which, on one hand, can serve as a benchmark for the properness of the associated idealized forcing of positive Borel sets ordered by inclusion, and, on the other hand, unifies many of the results that can be found in Zapletal's book \cite{Zapletal2008}. We show that under mild absoluteness assumptions, it implies properness, various dichotomy theorems, and closure under well-ordered unions in the Solovay model and under $\AD^+$, among other things. All major classes of proper idealized forcings studied in the book have this property. Further, we use this viewpoint to answer a question of Khomskii by showing that the naive idealized forcing for adding an eventually different real or a refining real is not proper below some condition. We also answer a question related to the definability of $\sigma$-ideals generated by Borel sets due to Kanovei, Sabok, and Zapletal. 
\end{abstract}

\maketitle

\section{Introduction}

\emph{Idealized forcing}, as developed by Zapletal in \cite{Zapletal2004}, \cite{Zapletal2008}, studies forcing notions consisting of $\mathcal{I}$-positive Borel sets ordered by inclusion, where $\mathcal{I}$ is a $\sigma$-ideal on a Polish space. This framework unifies many classical forcing notions in set theory -- Cohen, random, Sacks,  and Miller forcing, among others — and connects the study of cardinal invariants of the continuum with descriptive set theory. The case that is by far the best understood is where $\mathcal{I}$ is generated by $F_\sigma$ sets, or in other words, $\mathcal{I}$ is $\sigma$-generated by closed sets. This builds on work of Solecki \cite{Solecki1994}. In this case, the forcing is automatically proper, has the continuous reading of names, preserves Baire category, and, maybe most surprisingly, every strict intermediate extension is an extension by a single Cohen real (see \cite[Section 4.1]{Zapletal2008}). In particular, the forcing is either minimal or adds a Cohen real. To our knowledge, essentially no non-trivial general result is known in the next more complicated case where $\mathcal{I}$ is $\sigma$-generated by $G_\delta$ sets. In practice, these two classes cover the majority of naturally appearing idealized forcings. One of the main achievements of \cite{Zapletal2004} was in showing that for a wide class of cardinal invariants $\mathfrak{x}$, an inequality of the form $\mathfrak{x} < \mathfrak{c}$, if forceable at all, holds in the Sacks model. The upshot of this is a confirmation of the idea that Sacks forcing is optimal for adding a new real without affecting much else. Similar statements can be made for Miller forcing and adding an unbounded real, and many other examples exist (see \cite[p.94, p.97]{Zapletal2004}). 

Cardinal characteristics of the continuum, or their combinatorial reformulations, often directly give rise to an idealized forcing that, if all works out, should be optimal for increasing it in this sense. More precisely, given a Borel relation $R$ on the reals, the ideal $\mathcal{I}_R$ naturally associated to it consists of sets that are not $R$-$\omega$-dominating (see Definition~\ref{def:IR}). The corresponding idealized forcing $\mathbb{P}_R$ is the naive forcing designed precisely to add an $R$-dominating real, i.e. a real $d$ such that $x R d$, for all $x$ in the ground model. Adding an $R$-dominating real is exactly what needs to be done to increase the least size of an $R$-unbounded family, and the vast majority of cardinal invariants are of this form. A central question though is whether this forcing is proper in the first place. Many examples are of course known and \cite{Zapletal2008} provides a supply of arguments, but no clear general picture seems to emerge. A broadly applicable criterion for properness seems to be lacking. This also extends to the various other properties studied in \cite{Zapletal2008}, such as descriptive dichotomies, or the situation in the Solovay model and under determinacy. While there is a sense that the proofs do resemble each other often, no general view is presented that can explain this and one is left with the feeling that, regarding such fundamental questions, the theory of idealized forcing is rather chaotic. The main purpose of our article is to shine new light on this by giving a viewpoint that can unify many of the original results and hopefully make them more accessible. By doing this we are also answering some questions that have been left open in the literature.

The central notion of our paper is that of \emph{universal domination}, or more generally, \emph{universal pseudo-genericity}, and the resulting \emph{universality property} of an ideal (see Definition~\ref{def:univ} and \ref{def:univprop} for a precise definition). We say that a real $x$ is $\mathcal{I}$-pseudo-generic over a model $M$ if it avoids every $\mathcal{I}$-small Borel set in $M$. It is universally $\mathcal{I}$-pseudo-generic if for every forcing notion $\mathbb{Q} \in M$, there exists a $\mathbb{Q}$-generic filter $G$ over $M$ such that $x$ remains pseudo-generic over $M[G]$. The ideal $\mathcal{I}$ is said to have the \emph{universality property} if every generic pseudo-generic real over a countable model $M$ is already universally pseudo-generic. If $\mathcal{I} = \mathcal{I}_R$, for a relation $R$, we also speak of $R$ having the universality property. Our original motivation for universality comes from upcoming work \cite{SchilhanCM}, but it turns out that many interesting properties of idealized forcing are consequences of this notion. Specifically, assuming that $\mathcal{I}$ is provably $\mathbf{\Delta}^1_2$ on $\mathbf{\Sigma}^1_1$ (see Definition~\ref{def:provablydef}) and has the universality property, we have the following:

\begin{itemize}
\item $\mathbb{P}_{\mathcal{I}}$ is proper. In fact, the \emph{weak universality property} of $\mathcal{I}$ is equivalent to properness of $\mathbb{P}_{\mathcal{I}}$ (Theorem~\ref{thm:weakunivproper}).
\item Every generically added $\mathcal{I}$-pseudo-generic real is $\mathbb{P}_{\mathcal{I}}$-generic over some forcing extension of the ground model, in other words, is \emph{virtually $\mathbb{P}_{\mathcal{I}}$-generic} (Theorem~\ref{thm:virtgen}).
\item Every analytic set is either contained in a small Borel set or contains a positive Borel set. Assuming large cardinals, this can be extended to all universally Baire sets (Theorem~\ref{thm:dichotomy}).
\item In the Solovay model or assuming $\AD^+$, every set is either in $\mathcal{I}$ or contains an $\mathcal{I}$-positive Borel set, $\mathcal{I}$ is closed under well-ordered unions, and every total relation on the reals can be uniformized by a Borel function on an $\mathcal{I}$-positive set (Theorem~\ref{thm:AD+} and \ref{thm:solovay}).
\end{itemize}

Moreover, preservation properties of $\mathbb{P}_{\mathcal{I}}$, such as preserving Baire category, admit reformulations in terms of universality-like statements (see Section~\ref{sec:preserv}).

A major appeal of universality lies in the fact that it can often be verified purely through combinatorial arguments about the relation $R$, without needing to analyse the structure of conditions in $\mathbb{P}_{\mathcal{I}}$. Of course, one also needs to know that $\mathcal{I}$ is $\mathbf{\Delta}^1_2$ on $\mathbf{\Sigma}^1_1$. Assuming sufficient large cardinals, this assumption can usually be eliminated though. Also it turns out that, on its own, universality is still a very good test question. An open problem of Khomskii (see \cite[Example 2.5.6]{Khomskii}) asks whether the two forcing notions derived from the \emph{eventual difference} and the \emph{refining}/\emph{unsplitting}/\emph{reaping} relation are proper. Using a combinatorial construction, we show that these do not have the universality property and then turn the argument into non-properness of the forcing below a particular condition (see Section~\ref{sec:nonexample}).

On another note, let us remark that no natural example of a $\sigma$-ideal generated by Borel sets is known that is not $\mathbf{\Delta}^1_2$ on $\mathbf{\Sigma}^1_1$. In fact, Question 5.14 of Kanovei, Sabok, and Zapletal \cite{KanoveiSabokZapletal} even asks whether every such ideal is a \emph{game ideal}, a technical strengthening of this. It also asks whether every $\mathbf{\Delta}^1_2$ on $\mathbf{\Sigma}^1_1$ $\sigma$-ideal is a game ideal. We answer these questions in the negative. We show that there is a $\sigma$-ideal induced by an $F_\sigma$ relation, in particular it is $\sigma$-generated by $G_\delta$ sets, which gives rise to a proper forcing but is not $\mathbf{\Delta}^1_2$ on $\mathbf{\Sigma}^1_1$. Another such $F_\sigma$ relation does induce a $\mathbf{\Delta}^1_2$ on $\mathbf{\Sigma}^1_1$ $\sigma$-ideal but not a game ideal (Theorem~\ref{thm:univnotimplydef}).

All major examples of proper forcing in \cite{Zapletal2008} come from $\sigma$-ideals with the universality property and this will be the topic of Section~\ref{sec:examples}. These include the coanalytic \emph{porosity ideals} and \emph{analytic $P$-cover ideals} in the sense of \cite{FarahZapletal2006}.  As concrete applications, we establish universality of several natural idealized forcings for increasing cardinal invariants appearing in Cichoń's diagram. Specifically, we consider \emph{dominating forcing} induced by the well-known eventual dominance relation $<^*$, a slalom-based eventually different forcing, and a slalom-based localization forcing.

The paper is organized as follows. In the next section, we cover preliminaries on models of set theory, absoluteness, definability, and idealized forcing. In Section 3, we introduce the main notion of universal domination and develop its basic theory, including the equivalence with properness, and virtual genericity. We also give game characterizations of independent interest. In Section 4, we discuss preservation results, and in Section 5, we study the connection of universality to dichotomy theorems. Section 6 contains all the positive examples and shows how to view the results of \cite{Zapletal2008} within our framework. Section 7 presents the negative results on ``true" eventually different forcing and refining forcing. We study refining forcing in a bit more depth and show that under mild regularity assumptions it necessarily collapses the continuum to the dominating number $\mathfrak{d}$.

\section{Preliminaries}

\subsection{Models and absoluteness}\label{sec:modabs}

Throughout the paper we will say that a set or class $M$ is a model of set theory if $M$ is transitive and $(M,\in)$ models a sufficiently large finitely axiomatizable fragment of $\ZF + \DC$ including Powerset.\footnote{For instance, restricting $\ZF + \DC$ to only $\Sigma_{2026}$-formulas will be more than sufficient for everything in this paper.} 
In slight abuse of notation, we allow ourselves frequently to consider $M$ which are countable elementary submodels of some large fragment of the universe, so not outright transitive. In this case, one should identify $M$ with its transitive collapse, but to keep the notation as simple as possible we will never make this explicit. As such, we will often speak of forcing extensions $M[G]$ of $M$, and so on, when technically this only is defined for transitive models. This is fairly standard and no confusion or substantial issues shall ever arise from this.

We assume that the reader is familiar with the way various objects of descriptive set theory are relativized to models of set theory. When we speak about a particular Polish space $X$ we think of a fixed code of such a space and saying that $X \in M$ means that $M$ contains said code (which is a real). Similarly, a real $x \in X$ is a member of $M$ if an appropriate description of $x$ relative to the coding of $X$ is available to $M$, and so on.\footnote{E.g. see the expositions given in \cite{KanoveiSabokZapletal} or \cite{Moschovakis}.} We will make use of the common notation $A^{M}$ when we want to emphasize that the definition of the object $A$ is to be interpreted in $M$. Recall:

\begin{fact}[Mostowski Absoluteness]\label{fact:mostowskiabs}
    Let $M \subseteq N$ be (transitive) models of set theory, $x \in \omega^\omega \cap M$, and $\varphi(v)$ a $\Sigma^1_1$ formula. Then $M \models \varphi(x)$ iff $N \models \varphi(x)$.
\end{fact}

As a useful consequence, whenever $B \in M$ is a Borel set coded by two distinct reals $x_0, x_1 \in M \cap \omega^\omega$, these codes will still be equivalent in $N$. In particular, there is never any ambiguity in what the reinterpretation of a Borel set $B$ should be in any larger model. Similarly to the above, whenever we write $B \in M$ we treat $B$ as a Borel set in $V$ that has a code in $M$ and we have that $B^M = B \cap M$. This should all be fairly standard.

\begin{fact}[Shoenfield Absoluteness]
    Let $\omega_1^N \subseteq M \subseteq N$ be (transitive) models of set theory, $x \in \omega^\omega \cap M$, and $\varphi(v)$ a $\Sigma^1_2$ formula. Then $M \models \varphi(x)$ iff $N \models \varphi(x)$.
\end{fact}

It follows that there is an unambiguous way to reinterpret an analytic set in a forcing extension.

\subsection{Idealized forcing}

Recall that a $\sigma$-ideal on $X$ is a collection $\mathcal{I} \subsetneq \mathcal{P}(X)$ containing the empty set and closed under subsets and countable unions. We say that $A \subseteq X$ is $\mathcal{I}$-positive if $A \notin \mathcal{I}$. The collection of $\mathcal{I}$-positive sets is denoted $\mathcal{I}^+$.

\begin{definition}
    Let $\mathcal{I}$ be a $\sigma$-ideal on a Polish space $X$. Then $\mathbb{P}_\mathcal{I}$ is the forcing notion consisting of $\mathcal{I}$-positive Borel subsets of $X$ ordered by inclusion. 
\end{definition}

Since we will need to consider relativizations of ideals to different models $M$, let us adopt from now on the convention that whenever we speak of an ideal $\mathcal{I}$, we always think of a fixed definition for $\mathcal{I}$ that can have different interpretations $\mathcal{I}^M$. To simplify matters, we will also assume that the parameters in the definition of $\mathcal{I}$ are reals, and we write $\mathcal{I} \in M$ to say that $M$ contains these reals. This is a natural setting in which to state our results without introducing ambiguity. Also this covers essentially all interesting examples of ideals. Some straightforward generalisation would be to allow sets of ordinals as defining parameters when considering proper class (i.e. inner) models $M$; see more at the end of Section~\ref{sec:dichotomy}.

\begin{definition}\label{def:Ipseudogen}
    Let $\mathcal{I} \in M$ be a $\sigma$-ideal on a Polish space $X$ and let $x \in X$. Then we say that $x$ is \emph{$\mathcal{I}$-pseudo-generic} over $M$ iff $x \notin B$, for all Borel $B \in \mathcal{I}^M$. 
\end{definition}

\begin{fact}[see e.g. \cite{Zapletal2008}]
    Let $\mathcal{I} \in M$ be a $\sigma$-ideal on $X$ and let $G$ be $(\mathbb{P}_\mathcal{I})^M$-generic over $M$. Then there is some $x \in X$, such that
    \begin{enumerate}
        \item $\bigcap G = \bigcap_{p \in G} p =  \{ x\}$,
        \item $G = \{ p \in (\mathbb{P}_\mathcal{I})^M : x \in p \}$, 
        \item $x$ is $\mathcal{I}$-pseudo-generic over $M$.
    \end{enumerate}
    In this case, we say that $x$ is $\mathbb{P}_\mathcal{I}$-generic over $M$ and we write $M[G] = M[x]$.
\end{fact}

We also usually just say that $G$ is a $\mathbb{P}_\mathcal{I}$-generic filter over $M$, when technically $\mathbb{P}_\mathcal{I}$ should be relativized to $M$.

\begin{definition}
    Let $\mathcal{M}$ be a collection of models and $\mathcal{I}$ a $\sigma$-ideal on $X$. Then we say that $\mathbb{P}_{\mathcal{I}}$ is \emph{$\mathcal{M}$-proper} iff for each $M \in \mathcal{M}$, with $\mathcal{I} \in M$, and each $p \in (\mathbb{P}_{\mathcal{I}})^M$,  the set $\{ x \in p : x \text{ is } \mathbb{P}_{\mathcal{I}}\text{-generic over } M \}$ is $\mathcal{I}$-positive.
\end{definition}

Let us note that, at least when $\mathcal{P}(\omega^\omega) \cap M$ is countable, the set of generics over $M$ is a Borel set. In fact: 

\begin{lemma}\label{lem:genericsBorel}
    Let $M$ be a model, $\mathbb{P}, \dot x \in M$, where $M \models \text{``$\dot x$ is a $\mathbb{P}$-name for a real''}$, and assume $\mathcal{P}(\mathbb{P})^M$ is countable. Then the set $\{ \dot x^H : \text{$H$ is $\mathbb{P}$-generic over $M$} \}$ is Borel.
\end{lemma}

\begin{proof}
    By passing to a complete Boolean subalgebra of the regular open subsets of $\mathbb{P}$, we can assume without loss of generality that the evaluation of $\dot x$ uniquely determines the generic filter. To see that the set above is Borel simply note that it is the continuous injective image of a $G_\delta$ set in an appropriate Polish space of filters on $\mathbb{P}$.
\end{proof}

\begin{fact}[see e.g. \cite{Zapletal2008}]
    $\mathbb{P}_\mathcal{I}$ is proper (in the usual sense) iff it is $\mathcal{M}$-proper, where $\mathcal{M}$ is the collection of all countable elementary submodels of sufficiently large $H(\theta)$.
\end{fact}

Let us remark that properness can be developed without issues in $\ZF + \DC$. See \cite{AsperoKaragila} for more details. $H(\theta)$ must be defined as the collection of all $x$ such that $\theta$ does not inject into $\trcl(x)$. Clearly $H(\theta) \subseteq V_\theta$, so this is a set. We only consider $H(\theta)$ which are models of set theory in the sense we have stipulated above. In particular, $\theta$ should be a strong limit cardinal in order for the Powerset axiom to be satisfied. There are of course arbitrary large $\theta$, such that $H(\theta)$, and thus also its elementary submodels, are models of set theory in this sense.

$\mathcal{M}$-properness for the collection $\mathcal{M}$ of all countable models has sometimes been called \emph{non-elementary properness (nep)} or \emph{proper for candidates}, see e.g. \cite{Shelah2004}, \cite{Kellner2012}. 

In the study of cardinal invariants, many ideals of interest, or viewed from another side, many notions of pseudo-genericity, are obtained from a relation on the reals in the following way:

\begin{definition}\label{def:IR}
    Let $X$, $Y$ be Polish spaces and $R \subseteq X\times Y$ be a relation. We say that a set $A \subseteq Y$ is \emph{$R$-$\omega$-dominating}, if for any $C \in [X]^{\aleph_0}$ there is $y \in A$ with $\forall x\in C (xRy)$. $\sI_R$ consists of all sets $A \subseteq Y$ that are not $R$-$\omega$-dominating. We say that $y$ is \emph{$R$-dominating over $M$}, if for each $x \in X^M$, $xRy$.
\end{definition}

\begin{remark}
$\sI_R$ is a $\sigma$-ideal, exactly when for each $C \in [X]^{\aleph_0}$ there is $y \in Y$, such that $\forall x\in C (xRy)$, or in other words, exactly when $Y$ is itself $R$-$\omega$-dominating.
\end{remark}

\begin{lemma}\label{lem:pseudoisdom}
    Let $R \in M$ be Borel and let $y \in Y$. Then $y$ is $\mathcal{I}_{R}$-pseudo-generic over $M$ iff $y$ is $R$-dominating over $M$.
\end{lemma}

\begin{proof}
    Assume $y$ is pseudo-generic over $M$ and let $x \in M$. Then the Borel set $B = \{ z :  \neg (x R z) \}$ is in $(\sI_R)^M$. So $y \notin B$ and $x R y$. On the other hand, if $y$ is $R$-dominating over $M$ and $B \in (\sI_R)^M$, then there is some countable $C \in M$ so that $B \subseteq \{ z :  \exists x \in C \neg (x R z) \}$. By assumption, $y$ is not in the latter Borel set, so $y \notin B$.
\end{proof}

\begin{definition}
    We say that an ideal $\mathcal{I}$ on $X$ \emph{has a Borel base} if there is a Borel set $B \subseteq \omega^\omega \times X$ such that $\mathcal{I}=\{A \subseteq X : \exists x (A \subseteq B_x)\}$. 
\end{definition}

Note that for Borel $R$, $\mathcal{I}_R$ has a Borel base, and vice-versa, any $\sigma$-ideal defined by a Borel base is of this form. We will usually abbreviate $\mathbb{P}_{\sI_R}$ by just writing $\mathbb{P}_R$. 

\begin{exmp} Let $X, Y = \omega^\omega$.
    \begin{enumerate}
        \item $\mathbb{P}_{\neq}$ is equivalent to Sacks forcing,
        \item $\mathbb{P}_{\not>^*}$ is equivalent to Miller forcing,
        \item $\mathbb{P}_{\operatorname{split}}$, where $x \operatorname{split} y$ iff $\vert x \cap y \vert = \vert x \cap \omega \setminus y \vert = \omega$, is equivalent to splitting forcing, see \cite{Spinas2004},
        \item $\mathbb{P}_{<^*}$ is what we call \emph{dominating forcing}, and has been studied in \cite{BrendleHjorthSpinas},
        \item $\mathbb{P}_{\neq^*}$, where $x \neq^*y$ iff $\forall^\infty n (x(n) \neq y(n))$, will be called \emph{true eventually different forcing} (and is not proper, see Theorem~\ref{thm:edcoll}), 
        \item $\mathbb{P}_{\operatorname{ref}}$, where $x \operatorname{ref} y$ iff $y \subseteq^* x \vee y \subseteq^* \omega \setminus x$, will be called \emph{refining forcing}, (and is not proper, see Theorem~\ref{thm:refcoll}). It has been briefly studied in \cite{Spinas2008}.
    \end{enumerate}
\end{exmp}

(1), (2) and (3) are well-known to be proper. The properness of dominating forcing appears briefly in the literature, which wasn't known to us until recently. We include a new proof in Section~\ref{sec:dominatingforcing}, which is based on the technique we develop here. The properness of (5) and of (6) was an open question of Khomskii, see \cite[Example 2.5.6]{Khomskii}, which we answer negatively. 

\subsection{Definability of ideals}

\begin{definition}
    Let $\mathcal{I}$ be a $\sigma$-ideal on $X$ and $\mathbf\Gamma$ be a pointclass. Then $\mathcal{I}$ is said to be \emph{$\mathbf\Gamma$ on $\mathbf\Sigma^1_1$} if for each $\mathbf\Sigma^1_1$ set $A \subseteq \omega^\omega \times X$, the set $\{ z \in \omega^\omega : A_z \in \mathcal{I} \}$ is in $\mathbf\Gamma$.\footnote{Recall that $A_z$ denotes $A$'s ``vertical section'' at $z$, $A_z = \{ x \in X : (z,x) \in A \}$.}
\end{definition}

Recall that there is a universal lightface analytic set $U \subseteq \omega^\omega \times X$, which has the property that for any analytic $A \subseteq \omega^\omega \times X$, there is a continuous $f \colon \omega^\omega \to \omega^\omega$ with $A_z = U_{f(z)}$, for all $z$.\footnote{The usual definition of a universal analytic set, such as in \cite{Moschovakis}, is such that $(z,x) \in U$ if $x$ is in the analytic set coded by $z$ (a tree representation). For analytic $A \subseteq \omega^\omega \times X$, there is a fixed code and extracting a code for a particular section is a continuous operation.} Thus, if $\mathbf\Gamma$ is closed under continuous substitution, in the definition above, it suffices to consider a single such universal set. For the next definition, we fix such $U$.

\begin{definition}\label{def:provablydef}
For a projective pointclass $\mathbf{\Gamma}$, we say that \emph{$\mathcal{I}$ is provably $\mathbf{\Gamma}$ on $\mathbf\Sigma^1_1$} if there is a fixed $\mathbf\Gamma$-formula $\varphi$ such that for any countable model $M \ni \mathcal{I}$, $M$ contains the parameters of $\varphi$, and $$M \models \forall x \in \omega^\omega(\varphi(x) \leftrightarrow U_x \in \mathcal{I}).$$ We say that \emph{$\mathcal{I}$ is provably $\mathbf{\Delta}^1_n$ on $\mathbf\Sigma^1_1$}, if it is both provably $\mathbf{\Sigma}^1_n$ and provably $\mathbf{\Pi}^1_n$ on $\mathbf\Sigma^1_1$. More generally, we say that a statement holds \emph{provably}, if it holds in all countable models containing the relevant parameters. 
\end{definition}

Note that when something holds provably, then it holds in $V$ and all its forcing extensions since we can consider countable elementary submodels and their forcing extensions.

\begin{definition}
    Let $\mathcal{M}$ be a collection of models and $\mathcal{I}$ a $\sigma$-ideal on a Polish space. Then we say that 
    
    \begin{enumerate}
        \item $\mathcal{M}$ is \emph{$\Sigma^1_n$-correct} if $\Sigma^1_n$-formulas are absolute between $M \in \mathcal{M}$ and $V$,
        \item $\mathcal{M}$ is \emph{$\mathcal{I}$-correct} if for each $M \in \mathcal{M}$, with $\mathcal{I} \in M$, and each Borel $A \in M$, $A \in \mathcal{I}^M$ iff $A \in \mathcal{I}$,
        \item $\mathcal{M}$ is \emph{$\mathcal{I}$-upwards-correct} if $A \in \mathcal{I}^M$ implies $A \in \mathcal{I}$, for Borel $A$.
     \end{enumerate}   
\end{definition}

The notions above are defined to state the exact assumptions in some of the theorems. Note that when $R$ is defined by a $\mathbf\Sigma^1_1$ formula, then $\sI_R$ is provably $\mathbf\Sigma^1_2$ on $\mathbf\Sigma^1_1$, and in particular, all models are $\mathcal{I}_R$-upwards-correct by Mostowski absolutness. It is an empirical observation that in natural cases where $\mathbb{P}_R$ is proper, $\mathcal{I}_R$ is in fact provably $\mathbf\Delta^1_2$ on $\mathbf\Sigma^1_1$ so that all models are fully $\mathcal{I}_R$-absolute. It is unknown to us whether there is deeper fact behind this, but in most generality this is false, see Theorem~\ref{thm:univnotimplydef}. To know when an ideal is $\mathbf\Delta^1_2$ on $\mathbf\Sigma^1_1$ is one of the major problems of idealized forcing, at least in the setting where we do not want to assume large cardinals, and little seems to be known about it in general.

\begin{fact}[see e.g {\cite[Prop. 2.1.23]{Zapletal2008}}]
    If all models are $\mathcal{I}$-correct then $\mathcal{I}$ is $\mathbf\Delta^1_2$ on $\mathbf\Sigma^1_1$, and provably so, relative to a slightly stronger theory.
\end{fact}

The slightly stronger theory is a technical caveat.

\begin{proof}
    If $T$ is the defining theory of models of set theory that we set up at the beginning, then the slightly stronger fragment is $T$, together with the assertion that any real is contained in a countable transitive model of $T$. To know whether an analytic set is in $\mathcal{I}$, simply ask whether any/all transitive models of $T$ say so. This can be expressed in a $\mathbf\Delta^1_2$ fashion.
\end{proof}

The following notation will be useful very often: 

\begin{definition}
    Let $\mathcal{M}$ be a collection of models. Then we write $\mathcal{M}^+$ to denote the collection of all forcing extensions of models $M \in \mathcal{M}$.
\end{definition}

Recall that, by Shoenfield's absoluteness theorem, analytic sets can be reinterpreted in forcing extensions without the need to rely on a particular code. This idea ultimately culminates in the concept of a \emph{universally Baire set}. We omit the definition as it is only mentioned for context, but we refer to \cite{FengMagidorWoodin} for details.

\begin{fact}
    If every real has a sharp, then $\mathcal{M}^+$ is $\Sigma^1_2$-correct, where $\mathcal{M}$ is the collection of sufficiently elementary countable models.
\end{fact}

\begin{proof}
    By \cite[Theorem 3.4]{FengMagidorWoodin}, closure under sharps is equivalent to every $\mathbf\Sigma^1_2$ set being universally Baire. So suppose $\varphi$ is $\Sigma^1_2$ and that $M$ is elementary. Then there are trees $T,U \in M$ witnessing the universal Baireness. In a forcing extension $M[G]$, $p[T]$ still corresponds to the $\Sigma^1_2$-formula $\varphi$ (see the proof of \cite[Theorem 3.4]{FengMagidorWoodin}). So if $M[G] \models \neg \varphi(x)$, there is a branch $(x,y) \in [U] \cap M[G]$, so also $V \models \neg \varphi(x)$.
\end{proof}

\begin{definition}[Universally Baire absoluteness]\label{def:uBabs}
    We say that \emph{universally Baire absoluteness} holds if for any sufficiently elementary countable model $M$, a universally Baire set $A \in M$ and a forcing extension $M[G]$ of $M$, $M[G] \models \varphi$ iff $V \models \varphi$, where $\varphi$ is a projective formula allowing quantifiers over $A$. 
\end{definition}

Universally Baire absoluteness follows from large cardinals, see also \cite[Fact 1.4.11]{Zapletal2008}.

\section{Universal domination}

In the following, we introduce the central notion of this article. 

\begin{definition}\label{def:univ}
    Let $\mathcal{I} \in M$ be a $\sigma$-ideal on a Polish space and let $x$ be $\mathcal{I}$-pseudo-generic over $M$. Then we say that $x$ is \emph{universally $\mathcal{I}$-pseudo-generic over $M$} if for every forcing notion $\mathbb{Q} \in M$, there is a $\mathbb{Q}$-generic $G$ over $M$ such that $x$ remains $\mathcal{I}$-pseudo-generic over $M[G]$. Similarly, if $\mathcal{I} = \mathcal{I}_R$, we say that $x$ is \emph{universally $R$-dominating} over $M$.
\end{definition}

Note that the definition makes most sense when generics over $M$ exist at all. Thus we will usually only consider countable models.

\begin{definition}\label{def:univprop}
    Let $\mathcal{M}$ be a collection of countable models. A $\sigma$-ideal $\mathcal{I}$ is said to have the
    
    \begin{enumerate}
        \item \emph{universality property} for $\mathcal{M}$, if for every $M \in \mathcal{M}$, with $\mathcal{I} \in M$, every generic $\sI$-pseudo-generic real $x$ over $M$ is universally $\mathcal{I}$-pseudo-generic over $M$,
        \item \emph{weak universality property} for $\mathcal{M}$, if for every such $M$, every $\mathbb{P}_\sI$-generic real $x$ over $M$ is universally $\mathcal{I}$-pseudo-generic over $M$.
    \end{enumerate}
     If $\mathcal{M}$ is the collection of all countable models, then we simply say that $\mathcal{I}$ has the universality property or weak universality property, respectively. In the case of $\mathcal{I} = \mathcal{I}_R$, we will also say that $R$ has the (weak) universality property.
\end{definition}

The following is a useful observation for checking universality.

\begin{lemma}\label{lem:singleBorel}
     Let $\mathcal{I} \in M$ be upwards-absolute between forcing extensions of $M$ (e.g. $\mathcal{I}$ is provably $\mathbf\Sigma^1_2$ on $\mathbf\Sigma^1_1$). Then the following are equivalent: 

     \begin{enumerate}
        \item $x$ is universally $\mathcal{I}$-pseudo-generic over $M$.
        \item For every $\mathbb{Q} \in M$, and any $\mathbb{Q}$-name $\dot B$ for a Borel set in $\mathcal{I}^{M^\mathbb{Q}}$, there is a $\mathbb{Q}$-generic $G$ over $M$ such that $x \notin \dot B^G$.
    \end{enumerate}
\end{lemma}

\begin{proof}
    Of course (1) implies (2). So assume (2). To check (1), let $\mathbb{Q} \in M$ be arbitrary. Consider the forcing notion $\mathbb{Q}' = \mathbb{Q} * \Coll(\omega, \mathfrak{c}) \in M$.\footnote{$\Coll(\omega,\kappa)$ is the forcing notion consisting of finite partial functions from $\omega$ to $\kappa$ -- the natural forcing for making $\kappa$ countable. $\mathfrak{c}$ is the cardinality of the continuum.} Let $\dot B$ be a $\mathbb{Q}'$-name for the union of all Borel sets in $\mathcal{I}$ that lie in the intermediate $\mathbb{Q}$-extension. This is a countable union, thus $\dot B$ is a name for a Borel set in $\mathcal{I}$. Applying (2), we find a generic $G * H$ such that $x \notin \dot B^{G*H}$. It follows from the upwards-absoluteness that $x$ is pseudo-generic over $M[G]$.
\end{proof}

\subsection{On genericity}

The next two lemma's give some useful equivalent formulations of universality that let us drop or add the assumption of genericity for the objects appearing in the definition. The first one will later often be used implicitly. The second one is conceptually nice, as it is a simplification, but (seemingly) only applies under strong enough assumptions.

\begin{lemma}\label{lem:weak*}
 Let $x$ be generic and universally $\mathcal{I}$-pseudo-generic over $M$ and $\mathbb{Q} \in M$. Then there is a $\mathbb{Q}$-generic $G$ over $M$ as above, so that additionally $G$ is generic over $M[x]$ (for some forcing notion, not necessarily $\mathbb{Q}$).
\end{lemma}

In particular, this means that we can express universality of $x$ within $M[x]$, by asking whether an appropriate $G$ can be forced, rather than by quantifying over all generics over $M$.

\begin{proof}
Consider an extension $M[x][H]$ of $M[x]$ in which $\mathcal{P}(\mathbb{Q} \times \omega)^M$ is countable. Now the existence of $G$ as in the definition of universality can be expressed by a $\mathbf\Sigma^1_1$ sentence in $M[x][H]$. More specifically, it suffices to say that there is $G \subseteq \mathbb{Q}$ such that $G$ is $\mathbb{Q}$-generic over $M$ and for any $\mathbb{Q}$-name $\tau \in M$ of the form $\bigcup_{n \in\omega} A_n \times \{\check n\}$ for a real coding a Borel set in the extension's version of $\mathcal{I}$, $x$ is not in the Borel set coded by $\tau^G$. $\mathbb{Q}$, the set of dense subsets of $\mathbb{Q}$ in $M$, and the set of $\tau$ just described are all countable in $M[x][H]$, so it is routine to express this in a $\mathbf\Sigma^1_1$ way.

Since this $\mathbf\Sigma^1_1$ statement is true in $V$, by Fact~\ref{fact:mostowskiabs}, it also holds in $M[x][H]$, which is a forcing extension of $M[x]$.
\end{proof}

\begin{definition}
    Let us say that $\mathcal{I}$ has the \emph{total universality property} for $\mathcal{M}$ if every pseudo-generic over $M \in \mathcal{M}$ is universally pseudo-generic.
\end{definition}

\begin{lemma}
    Let $\mathcal{M}^+$ be $\Sigma^1_2$-correct. Then $\mathcal{I}$ has the total universality property for $\mathcal{M}$ iff it has the universality property for $\mathcal{M}$.
\end{lemma}

\begin{proof}
    Suppose there is $x$, pseudo-generic over $M$, but $\mathbb{Q} \in M$ witnesses that $x$ is not universal. Like above, let $M[H]$ be a forcing extension of $M$ in which $\mathcal{P}(\mathbb{Q} \times \omega)^M$ is countable. The statement that some pseudo-generic real $x$ over $M$ exists, such that for every $\mathbb{Q}$-generic $G$ over $M$, $x$ is not pseudo-generic over $M[G]$, is a true $\mathbf\Sigma^1_2$ statement with parameters in $M[H]$. Thus $M[H]$ satisfies it and contains such a real $x$. It follows that there is a generic pseudo-generic over $M$ that is not universal, so $M$ can't have the universality property.
\end{proof}

\begin{quest}
    Does the universality property of a Borel relation $R$ already imply the total universality property?
\end{quest}

\subsection{Simple examples}

$G_\delta$ relations correspond to $\sigma$-ideals that are $\sigma$-generated by closed sets. This is the most well-understood class of $\sigma$-ideals for idealized forcing. Among the most important consequences are properness, preservation of Baire category, every strict intermediate extension is a Cohen extension, and more, see \cite{Zapletal2008}.

\begin{lemma}
    Every $G_\delta$ relation has the total universality property.
\end{lemma}

\begin{proof}
    Let $R = \bigcap_{n \in \omega} O_n \subseteq X \times Y$, where each $O_n$ is open. Let $\mathbb{Q} \in M$ and let $d$ be $R$-dominating over $M$. We show that it suffices to let $G$ be $\mathbb{Q}$-generic over any countable model $M'$ extending $M$ and containing $d$. Namely let $\dot x \in M$ be any $\mathbb{Q}$-name for a real, $n \in \omega$, and $q \in \mathbb{Q}$ be arbitrary. Find a decreasing sequence $\langle q_i : i \in \omega \rangle \in M$ below $q$ interpreting $\dot x$ as some real $x \in X \cap M$.\footnote{For instance, for each $i$, $q_i$ decides $\dot x_n$ to be in a particular open set of diameter $< \frac{1}{i+1}$, in some complete metrization of $X$ given in $M$.} Since $d$ dominates $M$, $(x,d) \in O_{n}$ and there is a small enough neighborhood $U$ of $x$ so that $U \times \{d\} \subseteq O_{n}$. If $i$ is large enough such that $q_i \Vdash \dot x \in U$, then, in $M'$, $q_i \Vdash (\dot x, d) \in O_n$. Thus, by genericity over $M'$, $\dot x^G R d$.
\end{proof}

We remark that the argument above implicitly includes a proof that $\mathbb{P}_R$ preserves Baire category, which we will see in Corollary~\ref{cor:preserveBaire}.

\begin{lemma}
    There is a $G_{\delta\sigma}$ relation $R$ with the weak universality property but not the universality property.
\end{lemma}

\begin{proof}
    Consider $R \subseteq (2^\omega)^\omega \times (2^\omega)^\omega$, where $\bar x R \bar y$, if there is some $n$ so that $y_n \neq x_m$, for all $m$. $R$ does not have the universality property. Namely, for any countable $M$, let $\bar y$ enumerate the reals of $M$. Then $\bar y$ is $R$-dominating over $M$, but it can't be over any model that has a countable enumeration of $M$'s reals. 
    
    $R$ does have the weak universality property though. If $p \in \mathbb{P}_{R}$, then there must be some $n$, so that the projection $\pi_n p = \{a : \exists \bar x \in p (x_n = a) \}$ is uncountable. Find $q \leq p$, such that $\pi_n q \subseteq \pi_n p$ is uncountable and $q$ is uniformised on coordinate $n$, meaning that for each $a \in \pi_n q$, there is a unique $\bar x \in q$, with $x_n = a$.\footnote{Any Borel set can be uniformised on a perfect set.} Now note that a Borel subset $r \subseteq q$ is a condition if and only if it's projection $\pi_n r$ is still uncountable. It follows that $\mathbb{P}_{R}$ is equivalent to Sacks forcing and that if $\bar y$ is $\mathbb{P}_{R}$-generic with $q$ in it's generic filter, then $y_n$ is a Sacks real over $M$, i.e. $\mathbb{P}_{\neq}$-generic. The rest follows from $\neq$ having the universality property, as $\neq$ is $G_\delta$.
\end{proof}

We do not know whether the two universality notions can be separated with an $F_\sigma$ relation.

\begin{quest}\label{quest:Fsigmasepuniv}
    Is there an $F_\sigma$ relation with the weak universality property but not the universality property? What if the induced ideal is also provably $\mathbf{\Delta}^1_2$ on $\mathbf\Sigma^1_1$?
\end{quest}

\subsection{Properness}

\begin{thm}\label{thm:propertoweak}
    Let $\mathcal{I}$ be a $\sigma$-ideal on $X$ and $\mathcal{M}$ a collection of countable models with $\mathcal{M}^+$ being $\mathcal{I}$-upwards-correct (e.g. $\mathcal{I}$ is provably $\mathbf\Sigma^1_2$ on $\mathbf\Sigma^1_1$). If $\mathbb{P}_\sI$ is $\mathcal{M}$-proper, then $\sI$ has the weak universality property for $\mathcal{M}$.
\end{thm}

\begin{proof}
    Let $M \in \mathcal{M}$, $x$ be $\mathbb{P}_\sI$-generic over $M$ and $\mathbb{Q} \in M$ be arbitrary. Suppose towards a contradiction that there is no $\mathbb{Q}$-generic $G$ over $M$ so that $x$ is pseudo-generic over $M[G]$. In particular, one can't add such an object $G$ by forcing over $M[x]$. This can be expressed in the forcing language of $\mathbb{P}_\sI$, so there is some condition $p \in (\mathbb{P}_\sI)^M$, $x \in p$, forcing this. By $\mathcal{M}$-properness, $$q := \{y \in X : y \in p \wedge \text{$y$ is $\mathbb{P}_\sI$-generic over $M$} \}$$ is $\sI$-positive. The Borel set $q$ has a code in some forcing extension $M[H]$ of $M$. Let $G$ now be $\mathbb{Q}$-generic over $M[H]$. In $M[H][G]$, $q$ is $\mathcal{I}$-positive by upwards-correctness, so $q$ is a condition in $(\mathbb{P}_\sI)^{M[H][G]}$. Let $y \in q$ be $\mathbb{P}_\sI$-generic over $M[H][G]$. In particular, $y \in p$ and $y$ is $\mathbb{P}_\sI$-generic over $M$. But $G$ is generic over $M[y]$ (via some quotient forcing, not necessarily $\mathbb{Q}$ itself) and $y$ is pseudo-generic over $M[G]$. This poses a contradiction to the choice of $p$.
\end{proof}

Whenever $\sI$ is a $\sigma$-ideal with $\mathbb{P}_\sI$ ccc in $M$, pseudo-genericity over $M$ agrees with $\mathbb{P}_\sI$-genericity, so weak and total universality are equivalent. Moreover if $(\sI)^M \subseteq \sI$, i.e. we have upwards-correctness, then the set of $\mathbb{P}_\sI$-generics is $\mathcal{I}$-conull. We immediately obtain the following:

\begin{cor}
    Let $\mathcal{M}^+$ be $\mathcal{I}$-upwards-correct where $\mathbb{P}_\sI$ is ccc in every $M \in \mathcal{M}$ containing $\mathcal{I}$. Then $\sI$ has the total universality property for $\mathcal{M}$.
\end{cor}

\begin{thm}\label{thm:weakunivproper}
    Let $\mathcal{I}$ be a $\sigma$-ideal on $X$ and $\mathcal{M}^+$ be $\mathcal{I}$-correct (e.g. $\mathcal{I}$ is provably $\mathbf\Delta^1_2$ on $\mathbf\Sigma^1_1$). Then the following are equivalent:

     \begin{enumerate}
        \item $\mathbb{P}_\sI$ is $\mathcal{M}$-proper.
        \item $\sI$ has the weak universality property for $\mathcal{M}$.
        \item For each $\mathbb{P}_\mathcal{I}$-generic $x$ over $M \in \mathcal{M}$, there is a $\Coll(\omega, (2^{\mathfrak{c}})^M)$-generic $G$ over $M$, with $x$ pseudo-generic over $M[G]$.
    \end{enumerate}
\end{thm}

\begin{proof}
    Assume that $\sI$ has the weak universality property, let $M \in \mathcal{M}$ and let $p \in (\mathbb{P}_\sI)^M$ be arbitrary. We must show that $$q := \{ x \in X : x \in p \wedge \text{$x$ is $\mathbb{P}_\sI$-generic over $M$} \}$$ is $\sI$-positive. Let $x \in q$ be arbitrary. By the weak universality property, there is a generic extension $M[G]$ of $M$ in which $\mathcal{P}(\omega^\omega)^M$ is countable and over which $x$ is still pseudo-generic. In $M[G]$, the Borel set $q$ has a code. Since $x \in q$, $q \notin (\sI)^{M[G]}$. By absoluteness, $q \notin \mathcal{I}$. 
    For (3), note that to conclude properness, we only had to use weak universality applied to $\mathbb{Q}=\Coll(\omega, (2^{\mathfrak{c}})^M)$, so we get back full weak universality.
\end{proof}

\begin{cor}
   Let $\mathcal{I}$ be provably $\mathbf\Delta^1_2$ on $\mathbf\Sigma^1_1$ and have the universality property. Then $\mathbb{P}_{\mathcal{I}}$ is proper in every forcing extension.
\end{cor}

\begin{proof}
    Note that both the universality property and the statement that a particular pair of $\mathbf\Sigma^1_2$ and $\mathbf\Pi^1_2$ formulas renders $\mathcal{I}$ $\mathbf\Delta^1_2$ on $\mathbf\Sigma^1_1$ in any countable model can be expressed in a $\mathbf\Pi^1_2$ way. By Shoenfield-absoluteness, the same still holds in any forcing extension.
\end{proof}

In exactly the same vein: 

\begin{thm}\label{thm:univproper}
    Let $\mathcal{I}$ be a $\sigma$-ideal on $X$ and $\mathcal{M}^+$ be $\mathcal{I}$-correct. Then the following are equivalent:

     \begin{enumerate}
        \item $\sI$ has the universality property for $\mathcal{M}$.
        \item For any $M \in \mathcal{M}$ and $(\mathbb{P}, \dot x) \in M$, where $\Vdash_{\mathbb{P}} \text{``$\dot x$ is pseudo-generic''}$ in $M$, $\{ \dot x^H : \text{$H$ is $\mathbb{P}$-generic over $M$} \} \in \mathcal{I}^+$.
    \end{enumerate}
\end{thm}

\subsection{Virtual genericity}

\begin{definition}
   Let $\mathcal{I} \in M$ be a $\sigma$-ideal on $X$. Then $x \in X$ is said to be \emph{virtually $\mathbb{P}_\sI$-generic} over $M$ if there is some forcing extension $M[G]$ of $M$ so that $x$ is $\mathbb{P}_\sI$-generic over $M[G]$. Say that $\mathcal{I}$ has the \emph{virtual genericity property} for $\mathcal{M}$ if every generic $\mathcal{I}$-pseudo-generic over $M \in \mathcal{M}$ is virtually $\mathbb{P}_\sI$-generic over $M$.
\end{definition}

\begin{thm}\label{thm:virtgen}
     Let $\mathcal{I}$ be $\sigma$-ideal on $X$ and $\mathcal{M}$ a collection of countable models. If $\mathcal{I}$ has the universality property for $\mathcal{M}$ then $\sI$ has the virtual genericity property for $\mathcal{M}$.
\end{thm}

\begin{proof}
    Let $M \in \mathcal{M}$ and $x$ be $\mathbb{Q}$-generic and $\sI$-pseudo-generic over $M$, for some forcing notion $\mathbb{Q} \in M$. More specifically, there is a $\mathbb{Q}$-generic filter $H$ over $M$ and a $\mathbb{Q}$-name $\dot x \in M$ so that $x = \dot x^H$. Without loss of generality, by passing to a complete subforcing of $\mathbb{Q}$, we can assume that $\mathbb{Q}$ forces that the generic extension is generated by $\dot x$ (i.e. that $M[H] = M[x]$).
    Suppose towards a contradiction that $x$ is not virtually generic over $M$. In particular, one can't force an object $G$ over $M[x]$, so that $x$ is $\mathbb{P}_\sI$-generic over $M[G]$.\footnote{Note that in contrast to Lemma~\ref{lem:weak*}, in the definition of virtual genericity, $G$ has to be generic over $M[x]$ as $G$ and $x$ live in a common forcing extension of $M$.} This can be expressed in the forcing language of $\mathbb{Q}$, so there is some $s \in \mathbb{Q}$ forcing this over $M$.
    Applying universality, there is a forcing extension $M[G]$ of $M$ in which $\mathcal{P}(\mathbb{Q})^M$ is countable and over which $x$ is still pseudo-generic. Then the set $$p_s := \{ \dot x^K : s \in K \wedge \text{$K$ is $\mathbb{Q}$-generic over $M$}\}$$ is a Borel set with a code in $M[G]$. Moreover, $p_s \notin (\sI)^{M[G]}$ as $x \in p_s$. Let $y$ be $\mathbb{P}_\sI$-generic over $M[G]$, with $y \in p_s$. Since $M[G][y]$ is a forcing extension of $M$, $G$ is also generic over $M[y]$ (via some quotient forcing). But $y = \dot x^K$ for $K$ $\mathbb{Q}$-generic over $M$, $s \in K$. This poses a contradiction.
\end{proof}

\begin{thm}
     Let $\mathcal{M}$ be $\mathcal{I}$-correct, where $\mathcal{M}^+ \subseteq \mathcal{M}$ (e.g. $\mathcal{I}$ is provably $\mathbf\Delta^1_2$ on $\mathbf{\Sigma}^1_1$). Then $\mathcal{I}$ has the universality property for $\mathcal{M}$ if and only if the conjunction of the following holds:
     \begin{enumerate}
         \item $\mathbb{P}_\sI$ is $\mathcal{M}$-proper.
         \item $\sI$ has the virtual genericity property for $\mathcal{M}$.
     \end{enumerate}
\end{thm}

\begin{proof}
    Let $x$ be generic pseudo-generic over $M \in \mathcal{M}$ and $\mathbb{Q} \in M$. By (2), $x$ is $\mathbb{P}_\sI$-generic over some $M[H]$. We have that $M[H] \in \mathcal{M}$, so according to (1), there is a $\mathbb{Q}$-generic $G$ over $M[H]$ so that $x$ is still pseudo-generic over $M[H][G]$. But of course $G$ is also $\mathbb{Q}$-generic over $M$ and $x$ is pseudo-generic over $M[G]$.
\end{proof}

\begin{exmp}
    Every new generic real over $M$ is a Sacks real over some forcing extension of $M$, every generic unbounded real over $M$ is a Miller real over a forcing extension of $M$, every generic splitting real is virtually generic for splitting forcing, every generic dominating real is virtually generic for dominating forcing, etc.
\end{exmp}

One can view this as an argument for the canonicity of the forcings $\mathbb{P}_R$ for adding an $R$-dominating real.

\subsection{Game characterizations}

\begin{definition}
    Let $\mathcal{I}$ be a $\sigma$-ideal on $X$, $\mathbb{P}$ and $\mathbb{Q}$ forcing notions, and $\dot x$ a $\mathbb{P}$-name for an element of $X$. The game $\Game_{u}(\mathcal{I}, \mathbb{Q}, (\mathbb{P}, \dot x))$ is defined as follows:
        \begin{center}
            \begin{tabular}{c|c c c c c c}
            I &  $p_0, \dot B_0, A_0$ & & $p_2 \leq p_1, \dot B_1, A_1$ & & \dots\\
            II & & $p_1 \leq p_0, q_0 \in A_0$ & & $p_3 \leq p_2, q_0 \geq q_1 \in A_1$ &\\
\end{tabular}
        \end{center}
        The $\dot B_n$ are $\mathbb{Q}$-names for Borel sets in the ideal, $A_n$ are open dense subsets of $\mathbb{Q}$ and $p_n$, $q_n$ are conditions in $\mathbb{P}$ and $\mathbb{Q}$ respectively. For a given run of the game, let $G$ be the filter on $\mathbb{Q}$ generated by $\{q_n : n \in \omega \}$ and $H$ generated by $\{p_n : n \in \omega \}$. Player II wins if $\dot x^H \notin \bigcup_{n \in \omega} \dot B_n^G$.

        We simply write $\Game_{u}(\mathcal{I}, \mathbb{Q})$ when $\mathbb{P} = \mathbb{P}_{\mathcal{I}}$ and $\dot x$ is the name for the $\mathbb{P}_{\mathcal{I}}$-generic real. 
\end{definition}

\begin{thm}\label{thm:gameexactchar}
    The game $\Game_{u}(\mathcal{I}, \mathbb{Q},(\mathbb{P}, \dot x))$ is determined. Moreover, the following are equivalent: 
    
    \begin{enumerate}
        \item Player II has a winning strategy in $\Game_{u}(\mathcal{I}, \mathbb{Q}, (\mathbb{P}, \dot x))$.
        \item For any sufficiently elementary countable model $M$ containing $\mathcal{I}$, $\mathbb{Q}$, $\mathbb{P}$ and $\dot x$, and for any $\mathbb{P}$-generic $H$ over $M$, there is a $\mathbb{Q}$-generic $G$ over $M$, generic over $M[H]$, with $\dot x^H$ $\sI$-pseudo-generic over $M[G]$.
    \end{enumerate}
\end{thm}

\begin{proof}
   Assuming the statement in (2) holds, let's describe a winning strategy for II in $\Game_{u}$. Let $\theta$ be large enough and $M \preccurlyeq H(\theta)$ countable with $\mathcal{I},\mathbb{Q}, \mathbb{P}, \dot x \in M$. Then working in $M$ there is a $\mathbb{P}$-name $\dot{\mathbb{Q}}'$ for a forcing notion adding an object $G$ over $M^{\mathbb{P}}$ as claimed by (2). In other words, there is a $\mathbb{P} * \dot{\mathbb{Q}}'$-name $\dot G$ for such an object. By elementarity, the same holds in any larger elementary submodel $M' \supseteq M$. Along a run of the game, II constructs an auxiliary sequence $\langle M_{2n+1} : n \in \omega \rangle$ of increasing countable elementary submodels and a decreasing sequence $\langle (p_{2n+1}, \dot r_{2n+1}) : n \in \omega \rangle$ in $\mathbb{P} * \dot{\mathbb{Q}}'$. At each turn, $M_{2n+1}$ will contain all the previous models (and $M_1 \supseteq M$) and all the objects played by Player I so far. Player I having played $p_{2n}$ and $A_{2n}$, II finds $(p_{2n+1}, \dot r_{2n+1}) \leq (p_{2n}, \dot r_{2n-1})$ forcing ``$\check{q} \in \dot G$'' for a condition $q \in A_{2n}$ extending the previous condition $q_{n-1}$ and then plays $q_n \leq q$ and $p_{2n+1}$. Letting $M_\omega = \bigcup_{n \in \omega} M_{2n+1}$, II can clearly also ensure that $\{ (p_{2n+1}, \dot r_{2n+1}) : n \in \omega \}$ generates a $\mathbb{P} * \dot{\mathbb{Q}}'$-generic filter $K$ over $M_\omega$ and that $\{q_n : n \in \omega \}$ generates $\dot G^K$ (by possibly hitting more dense sets than required by Player I). Since $M_\omega$ is also elementary, it follows that II wins.

   Now suppose that (2) fails. This is witnessed by a model $M$ and there is a condition $p \in \mathbb{P}^M$ forcing that no object $G$ as claimed can be added over $M^{\mathbb{P}}$. We describe a winning strategy for I. Player I starts with $p_0 = p$ and plays fairly similar to II above, to which end there will be an elementary model $M_\omega$ and $\{p_n : n \in \omega\}$ generates a $\mathbb{P}$-generic filter over $M_\omega$. Moreover I can play the $A_n$ and $\dot B_n$ so that $\{q_n : n \in \omega\}$ becomes $\mathbb{Q}$-generic over $M_\omega$ and the $\dot B_n$ enumerate all $\mathbb{Q}$-names in $M_\omega$ for ideal Borel sets. If II wins, then $\dot x^H$ is pseudo-generic over $M_\omega[G]$. By the argument of Lemma~\ref{lem:weak*}, no loss of generality is lost assuming that $G$ is generic over $M[H]$. This yields a contradiction.
\end{proof}

\begin{thm}\label{thm:gamechar}
    Let $\mathcal{M}$ be the collection of countable elementary submodels of sufficiently large $H(\theta)$.
    \begin{enumerate}
        \item Player II has a winning strategy in all games $\Game_u(\mathcal{I}, \mathbb{Q}, (\mathbb{P}, \dot x))$, where $\dot x$ is a $\mathbb{P}$-name for an $\mathcal{I}$-pseudo-generic, iff $\mathcal{I}$ has the universality property for $\mathcal{M}$.
        \item Player II has a winning strategy in all games $\Game_u(\mathcal{I}, \mathbb{Q})$ iff $\mathcal{I}$ has the weak universality property for $\mathcal{I}$.
    \end{enumerate}
\end{thm}

Let us remark a few variations on the game $\Game_u$ above. For one, if $\mathcal{I}$ is upwards-absolute between forcing extensions (which is the case for most natural ideals) for the characterisation in Theorem~\ref{thm:gamechar} to hold, it would be sufficient for Player I to play only a single name for a Borel set $\dot B$ with II's winning condition becoming ``$\dot x^H \notin \dot B^G$". This is the argument of Lemma~\ref{lem:singleBorel}. 

Next, if $\mathcal{I}$ has a Borel base $(B_x)_{x \in \omega^\omega}$, Player I might play names $\dot z_n$ for reals instead of Borel sets and II wins if $x^H \notin \bigcup_{n \in \omega} B_{\dot z_n^G}$. This is an equivalent game. Clearly it is an easier game for II to win, and whenever I would play $\dot B_n$ in the original game, they can instead reserve infinitely many later rounds $i \in I_n \subseteq \omega$ where they play names $\dot z_i$ so that ``$\dot B_n \subseteq \bigcup_{i \in I_n} B_{\dot z_{i}}$" is forced and play enough dense sets $\dot A$ so that $\dot B_n^G \subseteq \bigcup_{i \in I_n} B_{\dot z_{i}^G}$ holds.

Finally, consider $\Game_{u'}(\mathcal{I}, \mathbb{Q}, (\mathbb{P}, \dot x))$: 

 \begin{center}
    \begin{tabular}{c|c c c c c c}
            I &  $p_0, \dot B_0$ & & $p_2 \leq p_1, \dot B_1$ & & $q_2 \leq q_1, \dot B_2$ \\
            II & & $p_1 \leq p_0, q_0$ & & $p_3 \leq p_2, q_1 \leq q_0$ & &\dots
    \end{tabular}
\end{center}

Here note that Player I alternates between extending $p_n$ and $q_n$. The winning condition for II is as in $\Game_u$. Then this game is equivalent to $\Game_u(\mathcal{I}, \mathbb{Q}, (\mathbb{P}, \dot x))$. There are several ways to see this, but the easiest is probably to just rerun the argument from Theorem~\ref{thm:gameexactchar}. There is no added difficulty in showing that Player I wins when (2) of Theorem~\ref{thm:gameexactchar} fails. When (2) holds, argue as in Theorem~\ref{thm:gameexactchar} and recall that $\mathbb{P} * \dot{\mathbb{Q}}'$ adding a $\mathbb{Q}$-generic is equivalent to the existence of a projection map $\pi \colon \mathbb{P}_\sI*\dot{\mathbb{Q}}' \to \mathbb{B}$, where $\mathbb{B}$ is the complete Boolean of which $\mathbb{Q}$ is a dense subset. Here, being a projection means that 

    \begin{enumerate}[(a)]
    \item $\forall r \leq p \in \mathbb{P}_\sI*\dot{\mathbb{Q}}'( \pi(r) \leq \pi(p))$ and
    \item $\forall p \in \mathbb{P}_\sI*\dot{\mathbb{Q}}'$ and $b \leq \pi(p)$, there is $r \leq p$ such that $\pi(r) \leq b$.\footnote{See e.g. \cite[Remark 5.3]{Cummings2010}, keeping in mind though that the remark is sloppily formulated as it cannot be guaranteed in general that the projection maps to the trivial condition, as required by \cite[Definition 5.2]{Cummings2010}.}
    \end{enumerate}

II plays similarly to before, except that in turns after which I extends in $\mathbb{Q}$, II plays $q \leq \pi(p_{n}, \dot r_{n})$, to which end any further extension $q'\leq q$ played by I can be met with $(p_{n+1}, \dot r_{n+1})$ such that $\pi(p_{n+1}, \dot r_{n+1}) \leq q'$. In the next turn, Player I plays $p_{n+2} \leq p_{n+1}$, thus still $\pi(p_{n+2}, \dot r_{n+1}) \leq q'$, posing no problem for II to continue in this fashion. In this way, as before, if II ensures that $G$ and $K$ are generic over $M_{\omega}$, they obtain that $\pi[K] \cap \mathbb{Q} = \dot G^K = G$, and $G$ is as claimed by (2).

\section{Preservation theorems}\label{sec:preserv}

In this section, we observe that preservation properties of $\mathbb{P}_{\mathcal{I}}$ can also be viewed as universality-like statements.

\begin{thm}\label{thm:preseration}
    Let $\mathbb{P}_{\mathcal{I}}$ be $\mathcal{M}$-proper and $\mathcal{M}^+$ be $\mathcal{I}$-correct (e.g. $\mathcal{I}$ is provably $\mathbf\Delta^1_2$ on $\mathbf\Sigma^1_1$), where $\mathcal{M}$ contains all countable elementary submodels. Let $\mathcal{J}$ be have a provable Borel base and let $A \in \mathcal{J}^+$. Then the following are equivalent: 
    \begin{enumerate}
        \item $\mathbb{P}_{\mathcal{I}}$ preserves that $A\in \mathcal{J}^+$.
        \item For any $M \in \mathcal{M}$ with $\mathcal{I}, \mathcal{J} \in M$, and any $p \in (\mathbb{P}_{\mathcal{I}})^M$, there is a $\mathbb{P}_{\mathcal{I}}$-generic real $x \in p$ over $M$ and a $\mathcal{J}$-pseudo-generic $y\in A$ over $M[x]$, with $x$ still $\mathcal{I}$-pseudo-generic over some model $M' \in \mathcal{M}^+$, $M' \supseteq M \cup \{y\}$.\footnote{Of course, if $y$ is generic over $M$, then $M[y] \in \mathcal{M}^+$ is the minimal such model. For instance, this happens when $\mathbb{P}_{\mathcal{J}}$ is ccc in $M$, since pseudo-genericity agrees with forcing genericity.}
    \end{enumerate}
\end{thm}

\begin{proof}
    (2) implies (1): Suppose $p_0 \Vdash \check{A} \subseteq \dot B$, where $\dot B$ is a $\mathbb{P}_{\mathcal{I}}$-name for a Borel set in $\mathcal{J}$. Let $N$ be a sufficiently elementary countable model containing $\mathcal{I}, \mathcal{J}, p_0, \dot B$. Let $p$ be the set of $\mathbb{P}_{\mathcal{I}}$-generics over $N$ contained in $p_0$ and let $M \ni N$ be a further countable elementary model. By (2), there is some $\mathbb{P}_{\mathcal{I}}$-generic $x \in p$ over $M$ and $y\in A$ pseudo-generic over $M[x]$ with $x$ still pseudo-generic over $M'\supseteq M \cup \{y\}$. Consider $q := \{ x' \in p : y \notin \dot B^{x'}\}$. We obtain a contradiction if we can show that $q \leq p$ is a condition in $\mathbb{P}_{\mathcal{I}}$, since then $q \Vdash \check{y} \notin \dot B$. But $q$ has a code in $M'$ and $x \in q$ is pseudo-generic over $M'$. The argument follows as usual: If $q \in \mathcal{I}$, then $M' \models q \in \mathcal{I}$, so $x$ can't be pseudo-generic.
    
    (1) implies (2): Suppose $M \in \mathcal{M}$ and $p \in (\mathbb{P}_{\mathcal{I}})^M$. Let $p' \leq p$ be the set of $\mathbb{P}_{\mathcal{I}}$-generics over $M$. Given a provable Borel base $\{ B_z : z \in \omega^\omega \}$ for $\mathcal{J}$, consider the name $\dot B$ for $\bigcup_{z \in M[\dot x]} B_z$, where $\dot x$ is name for the $\mathbb{P}_{\mathcal{I}}$-generic. By (1), there is some $y \in A$ and $q \leq p'$ such that $q \Vdash \check{y} \notin \dot B$. Let $M' \ni M \cup \{y, q\}$ be countable elementary and $x \in q$ be $\mathbb{P}_{\mathcal{I}}$-generic over $M'$. Then $y$ is $\mathcal{J}$-pseudo-generic over $M[x]$, since this is what $y \notin \dot B^x$ means. On the other hand, $x$ is $\mathcal{I}$-pseudo-generic over $M'$ as required.
\end{proof}

\begin{thm}
  Let $\mathcal{I}$, $\mathcal{J}$ and $\mathcal{M}$ be as before. Then the following are equivalent: 
    \begin{enumerate}
        \item For every $A \in \mathcal{J}^+$, $\mathbb{P}_{\mathcal{I}}$ forces ``$\check{A} \in \mathcal{J}^+$", i.e. \emph{$\mathbb{P}_{\mathcal{I}}$ preserves $\mathcal{J}$}.
        \item For any $M \in \mathcal{M}$ with $\mathcal{I}, \mathcal{J} \in M$, and any $p \in (\mathbb{P}_{\mathcal{I}})^M$, there is $B \in \mathcal{J}$ such that for all $y \notin B$, there is a $\mathbb{P}_{\mathcal{I}}$-generic $x \in p$ over $M$ with $y$ $\mathcal{J}$-pseudo-generic over $M[x]$, and $x$ $\mathcal{I}$-pseudo-generic over some $M' \in \mathcal{M}^+$, $M' \supseteq M \cup \{y\}$.
    \end{enumerate}
\end{thm}

\begin{proof}
    The direction from (2) to (1) follows easily from the previous theorem. Suppose (2) does not hold, witnessed by $M$ and $p$. Then the set $A$ of reals $y$, such that what is stated in (2) fails, is $\mathcal{J}$-positive. By the previous theorem, it is not the case that $\mathbb{P}_{\mathcal{I}}$ preserves $A$.
\end{proof}

We say that $\mathbb{P}$ preserves Baire category if it preserves every non-meager set. Similarly, $\mathbb{P}$ preserves Lebesgue measure if every non-null set is preserved.

\begin{cor}\label{cor:preserveBaire}
    Let $\mathcal{I}$ be as above and suppose that for any sufficiently elementary countable model $M$, any $\mathbb{P}_{\mathcal{I}}$-generic $x$ over $M$, and any Cohen real $c$ over $M[x]$, $x$ stays pseudo-generic over $M[c]$. Then $\mathbb{P}_{\mathcal{I}}$ preserves Baire category. Replacing ``Cohen" with ``random", we obtain that $\mathbb{P}_{\mathcal{I}}$ preserves Lebesgue measure.
\end{cor}

\section{The dichotomy theorems}\label{sec:dichotomy}

In \cite[Section 3.9]{Zapletal2008}, Zapletal considers the following three dichotomy style regularity properties of sets of reals.

\begin{definition}
    $\mathcal{I}$ satisfies the
    \begin{enumerate}
        \item \emph{first dichotomy} if every universally Baire set is either contained in a Borel set in $\mathcal{I}$ or contains an $\mathcal{I}$-positive Borel set,
        \item \emph{second dichotomy} if every universally Baire set is either in $\mathcal{I}$ or contains an $\mathcal{I}$-positive Borel set,
        \item \emph{third dichotomy} if every $\mathcal{I}$-positive analytic set contains an $\mathcal{I}$-positive Borel set. 
    \end{enumerate}
\end{definition}

\begin{thm}\label{thm:dichotomy}
    Let $\mathcal{M}$ be the collection of countable elementary submodels of sufficiently large $H(\theta)$ and suppose that $\mathcal{M}^+$ is $\mathcal{I}$-correct and that $\mathcal{I}$ has the universality property for $\mathcal{M}$. Then:
    
\begin{enumerate}
    \item Every analytic set is either contained in a small Borel set, or contains a positive Borel set, so $\mathcal{I}$ satisfies the third dichotomy.
    \item If $\mathcal{M}^+$ is $\Sigma^1_2$-correct (i.e. every real has a sharp), then every coanalytic set is either contained in a small Borel set, or contains a positive Borel set.
    \item If universally Baire absoluteness holds (see Definition~\ref{def:uBabs}), then $\mathcal{I}$ satisfies the first dichotomy.
\end{enumerate}
\end{thm}

\begin{proof}
    Let $A$ be analytic and $M$ be a countable elementary submodel containing $A$ and $\mathcal{I}$. If $A$ is contained in the (countable) union of small Borel sets in $M$, then $A$ is contained in a small Borel set. Otherwise, there is a pseudo-generic $x \in A$ over $M$. The existence of such $x$ can be formulated as a $\mathbf\Sigma^1_1$ statement in a forcing extension of $M$ in which $(\omega^\omega)^M$ is countable, so such $x$ can be added generically over $M$. Let $\mathbb{P} \in M$, $\dot x \in M$ a $\mathbb{P}$-name, such that this is forced of $\dot x$. The set $$B = \{ \dot x^H : H \text{ is $\mathbb{P}$-generic over } M \}$$ is then a Borel set, see Lemma~\ref{lem:genericsBorel}. By Mostowski-absoluteness, $B \subseteq A$. $B$ is $\mathcal{I}$-positive: Let $H$ be $\mathbb{P}$-generic over $M$ and $\mathbb{Q}\in M$ a forcing notion such that $B$ has a code in an extension by $\mathbb{Q}$. By universality, we can find a $\mathbb{Q}$-generic $G$ with $\dot x^H$ pseudo-generic over $M[G]$. So $B \notin \mathcal{I}^{M[G]}$ and followingly, by the correctness of $M[G]$, $B \notin \mathcal{I}$.

    (2) and (3) follow in exactly the same way, replacing the absoluteness argument accordingly.
 \end{proof}

The correctness assumptions can be eliminated if the models in consideration contain all ordinals. Recall that $A$ is called \emph{Suslin}, if there is a tree $T$ on $\omega \times \delta$, for some ordinal $\delta$, such that $A = p[T] = \{ x \in \omega^\omega : \exists w \in \delta^\omega \forall n \in \omega ((x\restriction n, w \restriction n) \in T)  \}$. We say that a sequence $\langle A_\alpha : \alpha < \gamma \rangle$ is \emph{uniformly Suslin} if there is $\langle T_\alpha : \alpha < \gamma \rangle$ such that $A_\alpha = p[T_\alpha]$, for each $\alpha<\gamma$.

\begin{thm}\label{thm:Suslin}
    Let $\mathcal{I}$ be provably $\mathbf\Sigma^1_2$ on $\mathbf\Sigma^1_1$ (e.g. defined by a Borel base) and have the universality property. Suppose that $\omega_1$ does not inject into the reals. Then
    
    \begin{enumerate}
        \item Every Suslin set is either contained in a small Borel set or contains a large Borel set. In particular, the first dichotomy holds.
        \item Any Suslin set can be uniformized by a Borel function on an $\mathcal{I}$-positive set. In other words, if $A \subseteq \omega^\omega \times \omega^\omega$ is Suslin with $\proj(A) \in \mathcal{I}^+$, then there is a Borel function $f \colon B \to \omega^\omega$, where $B \subseteq \proj(A)$, $B \in \mathcal{I}^+$, and for all $x \in B$, $(x,f(x)) \in A$.
        \item $\mathcal{I}$ is closed under well-ordered unions of uniformly Suslin sets.
    \end{enumerate}
\end{thm}

\begin{proof}
For (1), let $A = p[T]$ be Suslin and let $r$ be the real parameter defining $\mathcal{I}$. Suppose that $A$ is not contained in the union of all small Borel sets in $L[r,T]$. This is a countable union, thus a small Borel set $I$, as $\omega^\omega \cap L[r,T]$ is countable and $\mathcal{I}$ is upwards-absolute between $L[r,T]$ and $V$. Moreover, $\mathcal{P}(\omega^\omega) \cap L[r,T]$ is countable, since if $e$ enumerates the reals of $L[r,T]$, elements of $\mathcal{P}(\omega^\omega) \cap L[r,T]$ can be viewed as reals in $L[r,T,e]$, of which there are only countably many, since this is a model of $\ZFC$. Thus there is a generic $G$ for $\Coll(\omega, \omega^\omega \cap L[r,T])$ over $L[r,T]$. $I$ has a code in $L[r,T,G]$ and moreover there is a tree $S \in L[r,T,G]$ so that $p[S] = p[T] \setminus I$ in any extension.\footnote{$\omega^\omega \setminus I = p[T']$ for a tree $T'$ on $\omega$. Simply let $S$ consist of $(s,t,u)$, where $(s,t) \in T$ and $(s,u) \in T'$.} By assumption, $p[S]$ is non-empty and by a well-foundedness argument $S$ contains a branch in $L[r,T,G]$. This is all to say that there is a forcing notion $\mathbb{P}$ in $L[r,T]$, and a $\mathbb{P}$-name $\dot x$ for a real, such that 

\begin{enumerate}[(a)]
    \item $\mathcal{P}(\mathbb{P}) \cap L[r,T]$ is countable in $V$,
    \item $L[r,T]$ satisfies $\Vdash \text{``$\dot x$ is pseudo-generic"} \wedge \dot x \in p[T]$.
\end{enumerate}
The set
$$B = \{ \dot x^H : H \text{ is $\mathbb{P}$-generic over } L[r,T] \}$$ is then a Borel set and $B \subseteq p[T] = A$. Arguing as before, there is a forcing notion $\mathbb{Q} \in L[r,T]$, $\mathcal{P}(\mathbb{Q}) \cap L[r,T]$ countable, adding a code for $B$. If $B \in \mathcal{I}$, then by Shoenfield-absoluteness, $L[r,T,H] \models B \in \mathcal{I}$, for $\mathbb{Q}$-generic $H$. So suppose that $\mathbb{Q}$ forces that $B$ is small. In $L[r,T]$, let $M$ be a sufficiently elementary submodel with $\mathcal{P}(\mathbb{Q}) \cup \mathcal{P}(\mathbb{P}) \cup \omega^\omega \subseteq M$, reflecting that $\mathbb{Q}$ forces that $B$ is small. $M$ can be assumed to be countable in $V$. As before, we can apply universality to $M$ and obtain a contradiction.

For (2), let the $T$ be a Suslin representation of $A$ and, using (1), let $u$ be a Borel code for a large set $p \subseteq \proj(A)$. Whenever $x \in p$ is $\mathbb{P}_{\mathcal{I}}$-generic over $L[r,T,u]$, there is $y \in L[r,T,u][x]$ with $(x,y) \in A$, by another Suslin reflection argument. So let $\dot y \in L[r,T,u]$ be a $\mathbb{P}_{\mathcal{I}}$-name such that $p \Vdash (\dot x, \dot y) \in p[T]$, where $\dot x$ is a name for the generic real. It is immediate that $$f := \{ (\dot x^G, \dot y^G) : G \ni p \text{ is $\mathbb{P}_{\mathcal{I}}$-generic over } L[r,T,u] \}$$ is also Borel (see the proof of Lemma~\ref{lem:genericsBorel}). By universality (or simply properness for the model $L[r,T,u]$), the domain of $f$ is $\mathcal{I}$-positive.

(3) is as in \cite[Proposition 3.9.11]{Zapletal2008}. Suppose $\bar T = \langle T_\alpha : \alpha < \delta \rangle$ is given. Then $A = \bigcup_{\alpha < \delta} p[T_\alpha]$ is also Suslin. Suppose $A$ is large. By (1), there is a large Borel set $B \subseteq A$, say with code $u$. Then there is a pseudo-generic $x \in B$ over $L[r,u, \bar T]$ and for some $\alpha < \delta$, $x \in A_\alpha$. Arguing as above, $A_\alpha$ contains a large Borel set. \end{proof}

\begin{thm}\label{thm:AD+}
    Assume $\AD^+$ and let $\mathcal{I}$ be provably $\mathbf\Sigma^1_2$ on $\mathbf\Sigma^1_1$ (e.g. defined via a Borel base) and have the universality property. 
    
    \begin{enumerate}
        \item Every set of reals is either contained in a small Borel set, or contains a large Borel set.
        \item Every set of reals can be uniformized by a Borel function on an $\mathcal{I}$-positive set.
        \item $\mathcal{I}$ is closed under well-ordered unions.
    \end{enumerate}
\end{thm}

\begin{proof}
    
    $\AD^+$ implies that every true $\mathbf\Sigma^2_1$ statement has a Suslin witness, see e.g. \cite[Theorem 3.16]{Larson2025}. The existence of a counterexample to (1) can easily be expressed in a $\mathbf\Sigma^2_1$ way, so a Suslin counterexample must exist, which contradicts the previous theorem. Similarly then for (2). Closure under well-ordered unions follows essentially as in \cite[Proposition 3.9.18]{Zapletal2008}. Suppose $A = \bigcup_{\alpha < \delta} A_\alpha \in \mathcal{I}^+$. Without loss of generality, $\delta$ is regular uncountable (the singular case follows by induction). By \cite[Theorem 2.25]{Jackson2010}, there is a pre-well-ordering $\leq$ on $\omega^\omega$ of length $\delta$, such that the restriction of $\leq$ to any analytic set has at most countable length. Consider the set $P$ of pairs $(x,y)$, such that $x \in A_\alpha$ if and only if $y$ has position $\alpha$ in the pre-well-ordering $\leq$. By (2), there is a Borel function $f \colon B \to \omega^\omega$ such that $B \subseteq A$ is large and $(x,f(x)) \in P$, for all $x \in B$. $f[B]$ is analytic, whereby there are at most countably many positions of elements of $f[B]$ according to $\leq$. It follows that $B$ is contained in a countable union of $A_\alpha$'s, thus $A_\alpha \in \mathcal{I}^+$ for some $\alpha < \delta$.
\end{proof}

 Let us briefly recall the construction of a \emph{Solovay model}. Starting with a model $V$ of $\ZFC$ and a strong limit cardinal $\kappa \in V$, let $G$ be generic for the Lévy-collapse $\Coll(\omega, <\kappa)$.\footnote{The classical Solovay model uses an innaccessible cardinal, but our presentation makes the result also applicable to derived models of $\AD + V=L(\mathbb{R})$.} We then define $\mathbb{R}^*$ to be set of reals appearing in an intermediate extension $V[G \cap \Coll(\omega, \alpha)]$, $\alpha < \kappa$. To allow for the strongest statement possible, we say that the Solovay model is $V(\mathbb{R}^*)$, but any inner model between $L(\mathbb{R}^*)$ and $V(\mathbb{R}^*)$ works. The relevant part is that every set of reals in $V(\mathbb{R}^*)$ is $\infty$-Borel, and every set of ordinals appears in $V[x]$ for a real $x$. These are standard facts. Recall that $A$ is \emph{$\infty$-Borel} if there is a set of ordinals $S$, such that $A = \{ x : L[S,x] \models \varphi(S,x) \}$, for some formula $\varphi$ in the language of set theory.

\begin{thm}\label{thm:solovay}
    In the Solovay model, let $\mathcal{I}$ be provably $\mathbf\Sigma^1_2$ on $\mathbf\Sigma^1_1$ and have the universality property.
    
    \begin{enumerate}
        \item Every set of reals is either contained in a small Borel set or contains an $\mathcal{I}$-positive Borel set.
        \item Every set of reals can be uniformized by a Borel function on an $\mathcal{I}$-positive set.
        \item $\mathcal{I}$ is closed under well-ordered unions.
    \end{enumerate} 
\end{thm}

\begin{proof}
    (1) is similar to before. Let $r$ be as above and suppose that $A = \{ x : V[S,x] \models \varphi(S,x) \}$. If $A$ is not contained in the union of the small Borel sets in $V[r,S]$, then there is a forcing notion $\mathbb{P} \in V[r,S]$ of size $< \kappa$ and a name $\dot x$ so that over $V[r,S]$ it is forced that $\dot x$ is pseudo-generic and that $V[r,S,\dot x] \models \varphi(S,\dot x)$. In $V[G]$, $\mathcal{P}(\mathbb{P}) \cap V[r,S]$ is countable. If $B$ is the set of reals of the form $\dot x^H$, where $H$ is $\mathbb{P}$-generic over $V[r,S]$, then $B$ is as required. (2) and (3) are analoguous.
\end{proof}

Let us note that in the argument of Theorem~\ref{thm:Suslin} it suffices to use an $\infty$-Borel on $\mathbf{\Sigma}^1_1$ definition for $\mathcal{I}$ in order for the ideal to be sufficiently absolute between the various inner models. Both $\AD^+$ and the Solovay model imply that every set of reals is $\infty$-Borel so the assumption that $\mathcal{I}$ is provably $\mathbf\Sigma^1_2$ on $\mathbf\Sigma^1_1$, or that $\mathcal{I}$ is defined using a real paramater, is somewhat moot. One needs to be careful though about the meaning of universality, as it is in general dependent on the particular definition for $\mathcal{I}$ we choose and the models we consider should be able to carry the paramaters involved in the definition. It is certainly possible to state a more general result, restricting universality to models of the form $L[S]$, for $S$ a set of ordinals, and fixing a particular $\infty$-Borel definition. However, this deviates a bit from the exposition we give in this paper, in which we consider primarily countable models.

\section{Proper forcing notions}\label{sec:examples}

We now turn to examples, specifically considering the classes of ideals considered in \cite{Zapletal2008}. At the end, we will also consider a few concrete examples that are related to Cichoń's Diagram.

\subsection{Porosity ideals}

Porosity ideals were inspired by the notion of metric porosity and include ideals $\sigma$-generated by closed sets.

\begin{definition}
    Let $U$ be a countable collection of Borel subsets of a Polish space $X$. A \emph{porosity} is a monotone map $\pi \colon \mathcal{P}(U) \to \mathbf\Delta^1_1(X)$, i.e. $x \subseteq y \rightarrow \pi(x) \subseteq \pi(y)$. The associated \emph{porosity ideal} $\mathcal{I}_\pi$ is $\mathcal{I}_R$, where $R \subseteq \mathcal{P}(U) \times X$ is such that $x R y \leftrightarrow y \notin \pi(x) \setminus \bigcup x$.
\end{definition}

The porosity $\pi$ is said to be coanalytic, if $\{ (x,y) : y \in \pi(x) \}$ is a coanalytic subset of $\mathcal{P}(U) \times X$.

\begin{fact}[see {\cite[Theorem 3.8.7]{Zapletal2008}}]
    Let $\pi$ be a coanalytic porosity. Then $\mathcal{I}_\pi$ is provably $\mathbf\Pi^1_1$ on $\mathbf\Sigma^1_1$.
\end{fact}

The proof goes by first noting that the collection of \emph{porous sets}, i.e. those contained in a set $\pi(x) \setminus \bigcup x$, is $\mathbf\Pi^1_1$ on $\mathbf\Sigma^1_1$: A $\Sigma^1_1(x)$ set $A$ is porous if and only if it is contained in $\pi(x)$, for any $x$ such that $\forall u \in x (u \cap A = \emptyset)$. It is then shown that a $\Sigma^1_1(x)$ set is in the ideal if and only if it is contained in the union of all $\Sigma^1_1(x)$ porous sets. This gives a uniform $\Pi^1_1(x)$ definition. 

Note that for a coanalytic porosity $\pi$, the relation $R$ is only analytic and not Borel. The notions of being $R$-dominating and $\mathcal{I}_\pi$-pseudo-genericity still coincide though, since in any model, each porous analytic set is contained in a porous Borel set.\footnote{Of course, $\pi(x) \setminus \bigcup x$ is Borel in $V$ by definition, but this is not necessarily absolute to $M$, given the coanalytic definition of the porosity.} This follows from the First Reflection Theorem, see \cite[Theorem 35.10]{Kechris1995}. 

\begin{thm}
    Let $\pi$ be a coanalytic porosity. Then $\mathcal{I}_\pi$ has the total universality property and $\mathbb{P}_{\mathcal{I}_\pi}$ preserves Baire category.
\end{thm}

\begin{proof}
     Let $\mathbb{Q} \in M$ and $y$ be $R$-dominating over $M$. Let $G$ be $\mathbb{Q}$-generic over a countable model $M'$ extending $M$ and containing $y$. Let $p \in \mathbb{Q}$ and $\dot x \in M$, a $\mathbb{Q}$-name for an element of $\mathcal{P}(U)$, be arbitrary. Define $$x' := \{ u \in U : \exists q \leq p (q \Vdash u \in \dot x) \} \in M.$$ If $y \in \bigcup x'$, then there is $q \leq p$ be such that $q \Vdash u \in \dot x$, for some $u \in U$ with $y \in u$. In particular, $q \Vdash y \notin \pi(\dot x) \setminus \bigcup \dot x$. If $y \notin \bigcup x'$, then already $p \Vdash y \notin \pi(\dot x) \setminus \bigcup \dot x$, since $y \notin \pi(x') \setminus \bigcup x'$ and $p \Vdash \pi(\dot x) \subseteq \pi(x')$. By genericity, $y$ is $R$-dominating over $M[G]$. Preservation of Baire category follows from Corollary~\ref{cor:preserveBaire} and using Cohen forcing for $\mathbb{Q}$.
\end{proof}

\subsection{Farah-Zapletal ideals}

Many proofs of properness presented in \cite{Zapletal2008} are by means of a type of game originating in \cite{FarahZapletal2006} which we call the Farah-Zapletal game (specifically, see Sections 4.3-4.6).

\begin{definition}
    Let $\mathcal{I}$ be a $\sigma$-ideal on $X$ with a Borel base $B \subseteq \omega^\omega \times X$ and let $\dot x$ be a $\mathbb{P}$-name for an element of $X$, for some forcing notion $\mathbb{P}$. The \emph{Farah-Zapletal game} $\Game_{FZ}(B, (\mathbb{P}, \dot x))$ is defined as follows: 
    
     \begin{center}
    \begin{tabular}{c|c c c c c c}
            I &  $p_{0}, i_0, D_0$ & & $i_1, D_1$ & &\dots\\
            II & & $p_1 \leq p_{0}$ & & $p_2\leq p_1$ &
    \end{tabular}
\end{center}
$p_n$ ranges over $\mathbb{P}$, $D_n$ are dense open subsets of $\mathbb{P}$ and $i_n \in \omega^{<\omega}$. Given a run, let $z := (i_0,i_1,\dots)$. Player II wins if \begin{enumerate}
    \item for every $n$, there is $m$ with $p_m \in D_n$, and
    \item $\bigcap \{ O \subseteq X \text{ basic open} : \exists n (p_n \Vdash \dot x \in O) \} = \{ x \}$, where $x \notin B_z$.
\end{enumerate}

We call $\mathcal{I}$ a \emph{Farah-Zapletal ideal} if there is a provable Borel base $B$ for $\mathcal{I}$ such that for any pair $(\mathbb{P}, \dot x)$, where $\Vdash_{\mathbb{P}} \dot x \text{ is $\mathcal{I}$-pseudo-generic}$, Player II has a winning strategy in $\Game_{FZ}(B, (\mathbb{P}, \dot x))$.
\end{definition}

It is standard to note that the game $\Game_{FZ}(B, (\mathbb{P}, \dot x))$ is Borel in a suitable sense and thus determined. The Borel base witnessing being Farah-Zapletal isn't always the most natural one resulting from the definition of $\mathcal{I}$. Typically one defines $B$ so that I is forced to reveal extra information about the set they produce.

\begin{thm}\label{thm:FZ1}
    Let $\mathcal{I}$ be Farah-Zapletal as witnessed by a Borel base $B$. Then 
    \begin{enumerate}
        \item $\mathcal{I}$ has the universality property for the class of countable models satisfying that $\mathcal{I}$ is Farah-Zapletal as witnessed by $B$, e.g. elementary submodels of sufficiently large $H(\theta)$,
        \item $\mathcal{I}$ is provably $\mathbf\Delta^1_2$ on $\mathbf\Sigma^1_1$.
    \end{enumerate}
\end{thm}

\begin{proof}
    (1) is fairly similar to the argument with the games $\Game_u$. Let $\mathbb{P} \in M$ and $\dot x \in M$ a $\mathbb{P}$-name for a pseudo-generic real. $M$ contains a winning strategy for II in the game $\Game_{FZ}(B, (\mathbb{P}, \dot x))^M$. Moreover, this game being Borel, it is standard to note that being a winning strategy is upwards absolute by a ``well-foundedness of a Shoenfield-tree" argument.\footnote{The game can be viewed as a Gale-Steward game on a set $A \in M$, $A \subseteq M$, with payoff set a Borel subset of $A^\omega$. A strategy being winning is equivalent to saying that a particular tree is reverse ill-founded.} Let $\mathbb{Q} \in M$. By Lemma~\ref{lem:singleBorel} it suffices to consider a single $\mathbb{Q}$-name $\dot z \in M$. Suppose that $p \in \mathbb{P}$ forces that no $(M,\mathbb{Q})$-generic $G$ can be further added such that $x \notin B_{\dot z^G}$. Then let $\{p_n : n \in \omega\}$ be II's answer to a run where I starts with $p_0 = p$, plays all dense open sets in $M$ and step-by-step constructs a $\mathbb{Q}$-generic over $M$ deciding and playing out the initial segments of $\dot z$. We obtain a $\mathbb{P}$-generic $H \ni p$ over $M$ and a $\mathbb{Q}$-generic $G$ such that $\dot x^H \notin B_{\dot z^G}$. All of this can be done in a single forcing extension of $M$, contradicting the assumption on $p$.

    For (2) we give a game characterisation of analytic sets in the ideal. Let $C \subseteq X \times \omega^\omega$ be a closed set and consider the following game $\Game_g(B,C)$: 
    \begin{center}
    \begin{tabular}{c|c c c c c c}
            I &  $i_0$ & & $i_1$ & & \dots \\
            II & & $O_0, t_0$ & & $O_1 \subseteq O_0, t_1$ &
    \end{tabular}
\end{center}
$i_n \in \omega$, $t_n \in \omega^{<\omega}$, and $O_n$ are basic open subsets of $X$. We let $z := (i_0, i_1, \dots)$ and $w = {t_0}^\frown {t_1}^\frown\dots$. Player II wins if $\bigcap O_n = \{ x\}$ for some $x$ such that $(x,w) \in C$ and $x \notin B_z$.

\begin{claim}\label{claim:simplescram}
    II has a winning strategy iff the projection of $C$ to the first coordinate, $\operatorname{proj}(C)$, is in $\mathcal{I}^+$.
\end{claim}

\begin{proof}
    Clearly, if $\operatorname{proj}(C) \in \mathcal{I}$ then Player I has a simple winning strategy by playing out $z$ such that $\operatorname{proj}(C) \subseteq B_z$. On the other hand, if the projection of $C$ is large, then there is a pair $(\mathbb{P}, \dot x)$ such that $\Vdash \dot x \in \operatorname{proj}(C) \wedge \dot x \text{ is pseudo-generic}$: 
    
    Consider an elementary submodel containing $C$ and $\mathcal{I}$. The statement that there is $x \in \operatorname{proj}(C)$ that is pseudo-generic over $M$ is a true $\mathbf{\Sigma}^1_1$-statement that can be expressed in a forcing extension of $M$. Thus such a pair exists in $M$ and by elementarity also in $V$.

    Now for II to win, simply follow a strategy for $\Game_{FZ}(B, (\mathbb{P}, \dot x))$ adding to I's moves dense sets $D_n$ consisting of conditions deciding more and more of $\dot x$ and a witness for $\dot x \in \operatorname{proj}(C)$. In $\Game_g(B,C)$, II plays out the information that is eventually being decided.
\end{proof}

For a given closed set $C$, the statement that II has a winning strategy is $\Delta^1_2$, uniformily in a code for $C$ and $B$, by Borel determinacy. The universal $\Sigma^1_1$ set $A \subseteq \omega^\omega \times X$ is the projection of a $\Pi^0_1$ set $C \subseteq (\omega^\omega \times X) \times \omega^\omega$, so $\{ x \in \omega^\omega : A_x \in \mathcal{I}\}$ is $\Delta^1_2(c)$, for $c$ a code of $B$.\end{proof}

The game in (2) is strongly related to the notion of a \emph{game ideal} that has been developed in \cite{KanoveiSabokZapletal} and was inspired from the many Farah-Zapletal-style proofs in \cite{Zapletal2008}.

\begin{definition}[see {\cite{KanoveiSabokZapletal}}]
    $\mathcal{I}$ is a \emph{game ideal} if there is a provable Borel base $B$ and a Borel function $f \colon \omega^\omega \to X \times \omega^\omega$ (called a \emph{scrambling tool}) such that for any closed set $C \subseteq X \times \omega^\omega$, Player II has a winning strategy in the integer game where I produces $z \in \omega^\omega$, II produces $y \in \omega^\omega$, and II wins if $f(y) = (x,w) \in C$ and $x \notin B_z$.
\end{definition}

There is technically no assumption on $B$ being a Borel base \emph{provably} in \cite{KanoveiSabokZapletal}, as this is irrelevant to any consequence about $\mathcal{I}$ that is purely about the ideal and not its definition. But this is a necessary condition in the context of the next theorem.

The proof of Theorem~\ref{thm:FZ1} shows that Farah-Zapletal ideals are game ideals, for a particularly simple scrambling tool. The choice of the Borel base is important though. Let us refer to the situation in Claim~\ref{claim:simplescram} above with ``\emph{$\mathcal{I}$ is a game ideal with simple scrambling tool}". We note the following equivalence:

\begin{thm}
    The following are equivalent: 
    \begin{enumerate}
        \item $\mathcal{I}$ is Farah-Zapletal.
        \item $\mathcal{I}$ is a game-ideal with simple scrambling tool and $\mathcal{I}$ has the universality property for countable elementary submodels.
    \end{enumerate}
\end{thm}

We do not know if there is truly a difference between a game ideal with simple scrambling tool and a usual game ideal. Remarks made in \cite[5.1]{KanoveiSabokZapletal} express that masking II's point ``so that essentially nothing can be inferred about it" is crucial. On the other hand, it seems that in all examples of game ideals given the book, the Borel base is kept as the one naturally arising from the definition of the ideal or its choice is simply irrelevant. In the case of the games above, I is not disadvantaged by II's plays being masked, but rather by having to reveal specific information about their play by choice of a particular base.

\begin{proof}
We have shown that (1) implies (2). For the other direction, let $B$ witness that $\mathcal{I}$ is a game ideal with simple scrambling tool. Suppose Player I has a winning strategy in $\Game_{FZ}(B,(\mathbb{P}, \dot x))$, $\Vdash_{\mathbb{P}}\text{``$\dot x$ is pseudo-generic"}$, and let $p_0$ be the first play of I according to this strategy. Let $M$ be a countable elementary submodel containg all relevant information. Recall, by Theorem~\ref{thm:univproper}, that $A = \{ \dot x^H : \text{$H \ni p_0$ is $\mathbb{P}$-generic over $M$}\}$ is positive. In fact $A$ is the projection of a closed set $C$ described as follows: 

Let $\langle p_n, D_n : n \in \omega \rangle$ enumerate all conditions and dense open subsets of $\mathbb{P}$ in $M$. Then $(x,w) \in C$ iff $\langle p_{w(n)} : n \in \omega\rangle$ is decreasing in $\mathbb{P}$, $p_{w(n)} \in D_n$ and $p_{w(n)} \Vdash \dot x \in O$, for some basic open set $O \ni x$ of diameter $< \frac{1}{n+1}$ in some compatible metric. 

By assumption, II has a winning strategy in $\Game_g(B,C)$. II can apply this strategy to I's plays in $\Game_{FZ}(B,(\mathbb{P}, \dot x))$ ignoring the dense sets played by I. In $\Game_{FZ}(B,(\mathbb{P}, \dot x))$, at any stage of the game, Player II can repeat the same condition $p \in \mathbb{P} \cap M$ for as long as they want. Eventually though, for any dense set $D$ played by I, II's winning strategy in $\Game_g(B,C)$ will have to reveal $t_n \in \omega^{<\omega}$ such that $p_{t_n(i)} \in D$ and $p_{t_n(i)}$ extends the previous condition played by II, for some $i$. At that point II can switch to this condition. It is clear that this can be used to define a winning run of II against I's strategy.
\end{proof}

Since we have mentioned game ideals, we would like to point out a strong negative answer to Question 5.14 in \cite{KanoveiSabokZapletal} that has been missed by the authors.

\begin{thm}\label{thm:univnotimplydef}
    There is an $F_\sigma$ relation $R$ with the total universality property, and $\mathbb{P}_R$ proper, such that $\mathcal{I}_{R}$ is not $\mathbf\Delta^1_2$ on $\mathbf\Sigma^1_1$. Also, there is such a relation where $\mathcal{I}_R$ is $\mathbf\Delta^1_2$ on $\mathbf\Sigma^1_1$ (thus, $\mathbb{P}_R$ is proper) but not a game ideal. 
\end{thm}

\begin{proof}
    For the first statement, let $U = \{ z \in \omega^\omega :  \forall x\in \omega^\omega \exists y\in \omega^\omega \varphi(x,y,z)\}$ be a $\Pi^1_2$ set that is not $\mathbf{\Sigma}^1_2$, where $\varphi(x,y,z)$ is $\Pi^0_1$. Define the relation $(z_0,x, r) R (z_1, d)$ between $(\omega^\omega)^3$ and $(\omega^\omega)^2$ as $$z_0 \neq z_1 \vee ( r <^* d \wedge \exists y <^* d \varphi(x,y,z_0)).$$ 
    This is an $F_\sigma$ relation. Specifically note that, by compactness of $\{ y \in \omega^\omega : y <_n d \}$, ``$\exists y <^* d \varphi(x,y,z_0)$" can be expressed by saying $$\exists n \in \omega \forall k \in \omega \exists s \in \omega^k (s <_n d \wedge (x \restriction k, s, z_0 \restriction k) \in T ),$$ where $T$ is a tree representation of $\varphi$.\footnote{We write $s <_n d$ to mean $\forall m \in \dom(s), m\geq n(s(m) < d(m))$.}
    Note that $z \in U$ iff $\{z\} \times \omega^\omega \in (\mathcal{I}_R)^+$. Thus if $\mathcal{I}_R$ is $\mathbf\Pi^1_2$ on $\mathbf\Sigma^1_1$, then $U$ is $\mathbf{\Sigma}^1_2$. If $(z,d)$ is $R$-dominating over $M$, then this is either because $z \notin M$, or because $z \in M \cap U$ and $d$ is $<^*$-dominating over $M$. Both of these are universal properties (see Theorem~\ref{thm:domuniv} below for $<^*$). The resulting forcing notion is a lottery sum of Sacks and dominating forcing, so proper.

    For the second statement, first consider the set $W$ consisting of Borel codes $x \in \omega^\omega$ such that Player II has a winning strategy in the integer game with payoff set coded by $x$. This is a $\Delta^1_2$ set by Borel determinacy. Since there is no complete $\mathbf\Delta^1_2$ set, there must be another $\mathbf\Delta^1_2$ set $U \subseteq \omega^\omega$ not of the form $g^{-1}(W)$ for any Borel function $g \colon \omega^\omega \to \omega^\omega$.\footnote{We would like to thank Gabe Goldberg for pointing this out to us.}
    
    Just as before, we can write $U$ as $\{z \in \omega^\omega :  \forall x\in \omega^\omega \exists y\in \omega^\omega \varphi(x,y,z)\}$, where $\varphi$ is $\mathbf\Pi^0_1$, and define the relation $R$ as above. If $\mathcal{I}_R$ is a game ideal, it is easy to construct a Borel function $g$, arising from the scrambling tool, such that Player II has a winning strategy in the game with payoff set coded by $g(z)$ if and only if $\{z\} \times \omega^\omega \in (\mathcal{I}_R)^+$ -- that is, if and only if $z \in U$. This is impossible by assumption on $U$.
    
    To see that the ideal is $\mathbf\Delta^1_2$ on $\mathbf{\Sigma}^1_1$ we note again that $A \in \mathcal{I}^+$ iff $\proj(A)$ is uncountable or there is $z \in \proj(A)$ such that $z \in U$ and $A_z$ is a $<^*$-dominating family. Since both the ideal of countable sets and the ideal of non-dominating sets are $\mathbf\Delta^1_2$ on $\mathbf{\Sigma}^1_1$ (see more below), the result easily follows.
\end{proof}

We may still ask whether each ideal with Borel base is $\mathbf\Delta^1_2$ on $\mathbf\Sigma^1_1$ somewhere, and likewise about being a game ideal.

\begin{quest}\label{quest:univdef}
   For a $\sigma$-ideal with a Borel base, is its restriction to some positive set $\mathbf\Delta^1_2$ on $\mathbf\Sigma^1_1$? What if we assume the universality property or that it results in a proper forcing? What about $F_\sigma$ relations?
\end{quest}

A positive answer would be extraordinary as most of the paper would go through without the correctness assumptions.

\subsection{Analytic \texorpdfstring{$P$}{P}-cover ideals}

Analytic $P$-cover ideals were introduced in \cite{FarahZapletal2006}, were it is shown, implicitly, that they are Farah-Zapletal.\footnote{At least, if they are defined of the form $\operatorname{Exh}(\mu)$, see more below.}

\begin{definition}
    An ideal on $\omega$ is called a \emph{$P$-ideal} if for all $\{ x_n : n \in \omega \} \subseteq K$, there is $x \in K$ with $x_n \subseteq^* x$, for all $n$. We then define the associated \emph{$P$-cover ideal} $\mathcal{I}_{K}$ as $\mathcal{I}_R$, where $R \subseteq K \times \mathcal{P}(\omega)$ is defined as $x R y \leftrightarrow \vert x \cap y \vert < \omega$.
\end{definition}

$\mathbb{P}_{\mathcal{I}_K}$ adds a generic real $y \subseteq \omega$ that is almost disjoint from every ground-model element of the ideal $K$ (and infinite, if forcing below the condition $p = [\omega]^\omega$). The definition given in \cite{FarahZapletal2006} is slightly different, by using $x R y \leftrightarrow x \subseteq^* y$, therefore adding a real covering every ground-model ideal set, whence the name ``$P$-cover ideal". These are clearly equivalent presentations but ours is a bit more natural in light of the examples we will give below. Note that since $K$ is a $P$-ideal, sets of the form $\{ y : \vert y \cap x \vert = \omega \}$, for $x \in K$, form a base for $\mathcal{I}_K$, 

Recall the following well-known fact due to Solecki (see \cite{Solecki1999}):

\begin{fact}\label{fact:Solecki}
    $K$ is of the form $\operatorname{Exh}(\mu) = \{ x \subseteq \omega : \lim_{n \to \infty} \mu(x \setminus n) = 0 \}$, for a \emph{lower semicontinuous submeasure} $\mu$; that is a map $\mu \colon \mathcal{P}(\omega) \to [0, \infty]$ such that \begin{enumerate}
        \item $\mu(\emptyset) = 0$, $\mu(a) < \infty$ for $\vert a\vert <\omega$, $\mu(x) \leq \mu(y)$ for $x \subseteq y$, and $\mu(x \cup y) \leq \mu(x) + \mu(y)$,
        \item $\mu(x) = \lim_{n \to \infty} \mu(x \cap n)$, for every $x$.
    \end{enumerate}
\end{fact}
We note that $\mu(\bigcup x_n) \leq \sum_{n} \mu(x_n)$, which is used implicitly below. $P$-cover ideals arising from analytic $P$-ideals on $\omega$ satisfy the definition of being Farah-Zapletal witnessed by a Borel base that is $\Delta^1_1(\mu)$, for $\mu$ as above, see \cite[4.6.1]{Zapletal2008}. In most concrete cases, the submeasure is $\mu$ is fairly easy to describe from the definition of the ideal $K$ and the Borel base one obtains is a base \emph{provably}. On the other hand, a careful analysis of Solecki's proof might show that a fixed $\mu$ can be constructed in any model containing the analytic definition of $K$. In either of these cases, $\mathcal{I}_K$ is provably $\mathbf{\Delta}^1_2$ on $\mathbf{\Sigma}^1_1$ and a game ideal with simple scrambling tool.\footnote{In any case, we can't think of any meaningful application in which the distinction between $K$ defined analytically or as $\operatorname{Exh}(\mu)$ is substantial. It still follows that $\mathcal{I}_K$ is proper, satisfies the dichotomy type results, etc.}

\begin{thm}\label{thm:pcoveruniv}
    Let $K$ be an analytic $P$-ideal. Then $\mathcal{I}_K$ has the total universality property.
\end{thm}

\begin{proof}
     Let $M$ be a countable model, let $y \in [\omega]^\omega$ be $R$-dominating over $M$ and $\mathbb{Q} \in M$. Since $K$ is a $P$-ideal (and this is absolute) and by the argument of Lemma~\ref{lem:singleBorel}, note that it suffices to consider a single $\mathbb{Q}$-name $\dot x \in M$ for an element of $K$ and produce a $\mathbb{Q}$-generic $G$ with $\vert \dot x^G \cap y \vert < \omega$. Also, applying Fact~\ref{fact:Solecki} in $M$, let $\mu \in M$ with $M \models K = \operatorname{Exh}(\mu)$, and note that this is absolute to forcing extensions of $M$, being expressible as a $\mathbf\Pi^1_2$ statement.
    
   Work in $M$. For any $p \in \mathbb{Q}$ find a decreasing sequence $\langle p_n : n \in \omega \rangle$ below $p$ such that for each $n \in \omega$, $p_n$ decides $\dot x \cap n$ and a value $m_n > n$ such that \begin{equation}\label{eq:measbound} \mu(\dot x \setminus m_n) \leq 2^{-n}.\end{equation} Then let $x_p$ be the interpretation of $\dot x$ according to $\langle p_n : n \in \omega \rangle$ and note that $\mu(x_p \setminus m_n) \leq 2^{-n}$, for each $n$: If $\mu(x_p \setminus m_n) > 2^{-n}$, then there is a finite $a \subseteq x_p \setminus m_n$ with $\mu(a) > 2^{-n}$, so for some $k \geq n$, $p_k \Vdash \check{a} \subseteq \dot x \setminus m_n$ contradicting the bound forced by $p_n$. In particular, $x_p \in K$.

    \begin{claim}
        Let $D \subseteq \mathbb{Q}$ be dense open and $p \in \mathbb{Q}$. Then there is $x_{D,p} \in K$ so that for any $n \in \omega$, there is $m \geq n$ and $q \in D$, $q \leq p$, such that $x_{q} \setminus m \subseteq x_{D,p}$ and  $q \Vdash \dot x \cap m = x_{p} \cap m$.
    \end{claim}

    \begin{proof}
    Let $\langle p_n : n \in \omega \rangle$ be the decreasing sequence used to define $x_{p}$. We recursively construct a sequence $\langle q_{i}, n_i, m_i : i \in \omega \rangle$ as follows. Suppose $q_j, n_j, m_j$ have been defined for all $j < i$. Then let $n_i > i$, larger than any previous $n_j, m_j$, be large enough so that $$\sum_{j < i} \mu(x_{q_j} \setminus n_i) < 2^{-i}.$$
    
    Let $m_{i} = m_n$ as in (\ref{eq:measbound}) above, for $n = n_i$. Let $q_i \leq p_{m_{i}}$ in $D$ be arbitrary. We then have that 

    $$\mu(x_{q_i} \setminus m_i) \leq 2^{-n_i},$$ as $q_i \leq p_{m_i} \leq p_{n_i}$ and by the same argument as made directly before the claim. This finishes the construction of $\langle q_{i}, n_i, m_i : i \in \omega \rangle$. Finally, let $x_{D,p} := \bigcup_{i \in \omega} x_{q_i} \setminus m_i$. Then $x_{D,p} \in K$ as $$\mu(x_{D,p} \setminus m_i) \leq \sum_{j < i} \mu(x_{q_j} \setminus m_i) + \sum_{j \geq i} \mu(x_{q_j} \setminus m_j) \leq 2^{-i} + 2^{-i+1} \to 0.$$
    For any $i$ we have that $q_i \Vdash \dot x \cap m_i = x_{p} \cap m_i$ and $x_{q_i} \setminus m_i \subseteq x_{D,p}$. Since there are arbitrarily large $m_i$, this proves the claim.
\end{proof}
Outside of $M$, enumerate all dense sets in $M$ as $\langle D_i : i \in \omega \rangle$. We define a decreasing sequence $\langle q_i : i \in \omega \rangle$ in $\mathbb{Q}$ and $\langle n_i : i \in \omega \rangle$. Start with $q_0$ arbitrary and $n_0 = 0$. Suppose we have constructed $q_i$ and $n_i$. Then there is $n_{i+1} \geq n_i$ so that $y \cap x_{D_{i}, q_i} \subseteq n_{i+1}$, as $y$ is $R$-dominating over $M$. Moreover, making $n_{i+1}$ larger if necessary, according to the claim, we can find $q_{i+1} \leq q_i$ in $D_{i}$ so that $q_{i+1} \Vdash \dot x \cap n_{i+1} = x_{q_i} \cap n_{i+1}$ and $x_{q_{i+1}} \setminus {n_{i+1}} \subseteq x_{D_i,q_i}$. 

We find that $y \cap \dot x^G \subseteq n_1$, where $G$ is generated by $\{q_n : n \in \omega \}$. Namely, if $i \geq 1$, then $\dot x^G \cap [n_i, n_{i+1}) = x_{q_i}\cap [n_i, n_{i+1}) \subseteq x_{D_{i-1}, q_{i-1}}$. On the other hand, $y \cap x_{D_{i-1}, q_{i-1}} \subseteq n_i$, so $\dot x^G \cap [n_i, n_{i+1}) \cap y = \emptyset$.
\end{proof}

While this follows from Theorem~\ref{thm:FZ1} and the results of \cite{Zapletal2008}, let us include a direct proof that analytic $P$-cover ideals are $\mathbf{\Delta}^1_2$ on $\mathbf{\Sigma}^1_1$ for completeness. The argument is surprisingly similar to the one for universality which makes it natural to ask how closely connected universality is with the definability of the ideal (see Question~\ref{quest:univdef}).

\begin{thm}
    Let $K$ be defined as $\operatorname{Exh}(\mu)$ for a lower semicontinuous submeasure $\mu$. Then $\mathcal{I}_K$ is a game ideal with simple scrambling tool; in particular, it is provably $\mathbf{\Delta}^1_2$ on $\mathbf{\Sigma}^1_1$.
\end{thm}

\begin{proof}

Consider $B \subseteq ([\omega]^{<\omega} \times \omega)^\omega \times \mathcal{P}(\omega)$, where $(z,y) \in B$ iff, for $x := \bigcup_{n \in \omega} z(n)_0$, $\vert x \cap y \vert = \omega$ and for all $n \in \omega$,

\begin{enumerate}
    \item $z(n)_1 > n$,
    \item $z(n)_0 = x \cap z(n)_1$,
    \item $\mu( x \setminus  z(n)_1) \leq 2^{-n}$.\footnote{For each $n$, $z(n)$ is a pair and we write $z(n)_0$, $z(n)_1$, to refer to the first and second entry of $z(n)$.}
\end{enumerate}

This is a (provable) Borel base for $\mathcal{I}_K$. For a closed set $C \subseteq \omega^\omega \times \mathcal{P}(\omega)$, the game $\Game_{g}(B,C)$ translates to: 

\begin{center}
    \begin{tabular}{c|c c c c c c}
       I &  $a_0 \in [\omega]^{<\omega}, m_0\in\omega$ & & $a_1, m_1$ & & \dots\\
       II & & $s_0 \in 2^{<\omega},t_0 \in \omega^{<\omega}$ & & $s_1,t_1$ &\\
\end{tabular}
\end{center}

Player II wins if $m_n \leq n$, $a_n \neq a_{n+1} \cap m_n$, or $\mu( x \setminus m_n) > 2^{-n}$, for some $n$, or otherwise if $(y,w) \in C$ and $\vert y \cap x \vert < \omega$, where $x = \bigcup_{n \in \omega} a_n$, $y ={s_0}^\frown {s_1}^\frown \dots$ and $w = {t_0}^\frown {t_1}^\frown \dots$.

\begin{claim}
    Player II has a winning strategy iff $\operatorname{proj}(C) \notin \mathcal{I}_{K}$.
\end{claim}

\begin{proof}
    Clearly, if $\operatorname{proj}(C) \in \mathcal{I}_{K}$, Player I can win. So suppose $\operatorname{proj}(C) \notin \mathcal{I}_{K}$ but Player I still has a winning strategy. Let $M \preccurlyeq H(\theta)$, for large $\theta$, be a countable elementary submodel containing a winning strategy $\sigma$ for Player I. Since $\operatorname{proj}(C)$ is positive, there is $(y,w) \in C$ such that $\vert y \cap x \vert < \omega$ for all $x \in K \cap M$. We show that there is a winning run for II against $\sigma$, where they play out $y$ and $w$. In fact, the run will have the following particular form. 
    
    Given $\eta \in \omega^{\leq\omega}$, let $s(\eta) = \langle s_i(\eta) : i \in \omega \rangle$ and $t(\eta) = \langle t_i(\eta) : i \in \omega \rangle$ be a run of Player II, where first $s_i, t_i = \emptyset$ is played for $\eta(0)$-many rounds and then the first bit of $y$ and of $w$ is revealed, next $s_i, t_i = \emptyset$ is played for $\eta(1)$-many rounds and then the next bit of $y$ and of $w$ is revealed, and so on. In other words: $$s(\eta) = {\langle \emptyset \rangle^{\eta(0)\frown}} y(0)^\frown {\langle \emptyset \rangle^{\eta(1)\frown}} y(1)^\frown \dots,$$
     $$t(\eta) = {\langle \emptyset \rangle^{\eta(0)\frown}} w(0)^\frown {\langle \emptyset \rangle^{\eta(1)\frown}} w(1)^\frown \dots.$$

     If $\eta$ is infinite, then this defines a full run of the game. If $\eta$ is finite, then at the end of this process (i.e. from round $\vert \eta \vert + \sum_{i < \vert \eta \vert} \eta(i)$ onwards), $s_i, t_i = \emptyset$ is played again indefinitely. 
     
     Define $x(\eta)$ and $\bar m(\eta) = \langle m(\eta)_i : i \in \omega \rangle$ to be I's answer according to $\sigma$ to the run $s(\eta), t(\eta)$. For $\sigma$ to be winning, it must be the case that $\mu(x(\eta) \setminus m(\eta)_n) \leq 2^{-n}$, for every $n \in \omega$. Also note that $\langle x(\eta^\frown n) : n \in \omega \rangle \in M$, for every $\eta \in \omega^{<\omega}$, since we only need a finite initial segment of $y$ and $w$ to compute this from $\sigma$.

    \begin{claim}
        Work in $M$. Let $\eta \in \omega^{<\omega}$. Then there is $\tilde x(\eta) \in K$ such that for any $n \in \omega$, there are $n', m \geq n$ with $x(\eta^\frown n') \setminus m \subseteq \tilde x(\eta)$ and $x(\eta^\frown n') \cap m = x(\eta) \cap m$. 
    \end{claim}

    \begin{proof}
        Recursively define a sequence $\langle n_i, m_i : i \in \omega \rangle$ as follows. Given $n_j, m_j$ for $j < i$, we let $n_i > i$, greater than all previous $n_j, m_j$ and large enough so that $$\sum_{j < i} \mu(x(\eta^\frown n_j) \setminus n_i) < 2^{-i}.$$ Next, we define $m_i := m(\eta^\frown n_i)_{n_i}$. Let $\tilde x(\eta) = \bigcup_{i \in \omega} x(\eta^\frown n_i) \setminus m_i$ and as before, compute that $\tilde x(\eta) \in K$. By definition $x(\eta^\frown n_i) \setminus m_i \subseteq \tilde x(\eta)$. $x(\eta^\frown n_i) \setminus m_i = x(\eta) \cap m_i$ simply follows because for $\sigma$ to be winning, Player I must have revealed $x(\eta) \cap m_i$ at round $i$ already, which is before the point at which the runs induced by $\eta$ and $\eta^\frown n_i$ disagree, as $n_i > i$.
    \end{proof}

    We recursively define a sequence $\langle n_i,m_i : i \in \omega \rangle$, writing $\eta_i$ for $\langle n_j : j < i \rangle$, as follows. At step $i$, let $m_i$ and $n_i$ be larger than $n_j, m_j$, for all $j < i$, and such that \begin{enumerate}
        \item $y \cap \tilde x(\eta_i) \subseteq m_i$, 
        \item $x({\eta_i}^\frown n_i) \setminus m_i \subseteq \tilde x(\eta_i)$, 
        \item $x({\eta_i}^\frown n_i) \cap m_i = x(\eta_i) \cap m_i$.
    \end{enumerate}
    
    Let $\eta = \langle n_i : i \in \omega \rangle$. We claim that $y \cap x(\eta) \subseteq m_0$ and thus, the run $s(\eta), t(\eta)$ is winning for II. Namely, let $i \in \omega$ be arbitrary. Then $$x(\eta) \cap [m_i, m_{i+1}) = x(\eta_{i}^\frown n_i) \cap [m_i, m_{i+1}) \subseteq \tilde x(\eta_i) \cap [m_i, m_{i+1}).$$

    Since $y \cap \tilde x(\eta_i) \subseteq m_i$, we have that $y \cap x(\eta) \cap [m_i, m_{i+1}) = \emptyset$.
\end{proof}\end{proof}

In the following, we give some examples of forcing notions related to Cichoń's diagram arising as $\mathbb{P}_{K} \restriction p$ for an analytic $P$-ideal $K$ and a condition $p \in \mathbb{P}_{K}$. In the case of dominating forcing, we repeat the proof of universality because it is considerably shorter and conceptually easier than that Theorem~\ref{thm:pcoveruniv} and is what led to it in the first place. We also want to just highlight the nice combinatorial nature of these proofs.

\subsubsection{Dominating forcing}\label{sec:dominatingforcing}

The ideal $\mathcal{I}_{<^*}$ of non-eventually-dominating sets can be presented as a $P$-cover ideal arising from the analytic P-ideal $K$ on $\omega \times \omega$ consisting of $x$ such that $\{ m : (n,m) \in x\}$ is finite for every $n \in \omega$. Let $p$ consist of those $y \subseteq \omega \times \omega$ which are graphs of a total function from $\omega$ to $\omega$. Then there is a clear correspondence between $\mathcal{I}_{<^*}$ and $\mathcal{I}_K \restriction p$. Of course, dominating forcing provides a natural way to increase the bounding number $\mathfrak{b}$.

\begin{thm}\label{thm:domuniv}
    $<^*$ has the total universality property.
\end{thm}

\begin{proof}
    Let $M$ be a countable model, let $d$ be $<^*$-dominating over $M$ and $\mathbb{Q} \in M$. Recall, by Lemma~\ref{lem:singleBorel}, that it suffices to consider a single $\mathbb{Q}$-name $\dot x \in M$ for an element of $\omega^\omega$ and produce a $\mathbb{Q}$-generic $G$ with $\dot x^G <^* d$. This simplifies the proof slightly. 
    
    Work in $M$. For any $p \in \mathbb{Q}$, fix an interpreting function $g_{p}$ for $\dot x$ below $p$, i.e. a function $g \in \omega^\omega$ so that for every $n \in \omega$, there is $q \leq p$ so that $q \Vdash g \restriction n \subseteq \dot x$. 

    \begin{claim}
        Let $D \subseteq \mathbb{Q}$ be dense open and $p \in \mathbb{Q}$. Then there is $f_{D,p} \in \omega^\omega$ so that for any $n \in \omega$, there is $q \in D, q \leq p$ such that $g_{q} <_n f_{D,p}$  and $q \Vdash \dot x \restriction n = g_{p} \restriction n$.
    \end{claim}
    \begin{proof}
      For each $n$, pick $p_n \leq p$, with $p_n \Vdash \dot x \restriction n = g_{p} \restriction n$ and then let $q_n \leq p_n$ with $q_n \in D$. Simply define $f_{D,p}(n) = \max\{ g_{q_i}(n) : i \leq n\} +1$.
    \end{proof}

Outside of $M$, let $\langle D_i : i \in \omega \rangle$ enumerate all open dense subsets of $\mathbb{Q}$ in $M$. Construct sequences $\langle p_i : i \in \omega \rangle$ and $\langle n_i : i \in \omega \rangle$ as follows. We let $p_0 \in \mathbb{Q}$ be arbitrary and $n_0 = 0$. Next suppose $p_i$ and $n_i$ have been defined. Since $d$ is dominating over $M$, there is some $n_{i+1} > n_{i}$ so that $f_{D_i, p_i} <_{n_{i+1}} d$. We then let $p_{i+1} := q$ for $q$ as given by the claim for $p = p_i$, $D = D_i$ and $n = n_{i+1}$. Specifically, we have that $g_{p_{i+1}} <_{n_{i+1}} f_{D_i,p_i}$ and $p_{i+1} \Vdash \dot x \restriction n_{i+1} = g_{p_i} \restriction n_{i+1}$. 

The filter $G$ generated by $\langle p_i : i \in \omega \rangle$ is generic over $M$ since it hits every dense set. Moreover, $\dot x^G <_{n_{1}} d$. Namely if $m\geq n_1$ is arbitrary, say $m \in [n_i, n_{i+1})$ for $i \geq 1$, then $$ \dot x^G \restriction n_{i+1} = g_{p_{i}} \restriction n_{i+1} \text{ and } g_{p_{i}} <_{n_{i}} f_{D_{i-1}, p_{i-1}} <_{n_{i}} d.$$ All in all, $\dot x^G(m) = g_{p_i}(m) < d(m)$.
\end{proof}

It has been already shown in \cite{BrendleHjorthSpinas} that the ideal $\mathcal{I}_{<^*}$ is provably $\mathbf\Delta^1_2$ on $\mathbf\Sigma^1_1$. Specifically, it is shown that an analytic set is non-dominating if and only if it contains the branches of a tree-like structure on $\omega$ satisfying certain arithmetic conditions.

\begin{cor}
    $\mathbb{P}_{<^*}$ is proper, every generic dominating real is virtually generic for $\mathbb{P}_{<^*}$, and in the Solovay model or under $\AD^+$ every set of reals is either non-dominating or contains a Borel dominating family, and further, a closed dominating family by the results of \cite{BrendleHjorthSpinas}.
\end{cor}

Properness of dominating forcing has been implicitly mentioned in \cite[Example 8.2]{FarahZapletal2006}, see also \cite[4.6.6]{Zapletal2008}. It is claimed there, that \cite{BrendleHjorthSpinas} shows that the forcing is equivalent to Laver forcing, but \cite{BrendleHjorthSpinas} only contains an argument below the condition of strictly increasing functions.\footnote{This does of course not mean that the generic real of $\mathbb{P}_{<^*}$ is necessarily a Laver real.} We do not see an argument that this must hold everywhere, nor do we know how homogeneous or inhomogeneous $\mathbb{P}_{<^*}$ is.

\subsubsection{Slalom-based eventually different forcing}

\begin{definition}
    A slalom is a function $S \colon \omega \to [\omega]^{<\omega}$ so that $$\sum_{n = 1}^{\infty} \frac{\vert S(n) \vert}{n^2} < \infty.$$ For any slalom $S$ and $x \in \omega^\omega$, we write $x \notin^* S$ for $\forall^\infty n (x(n) \notin S(n))$.
\end{definition}

Consider the ideal $K$ on $\omega\times \omega$ consisting of sets of the form $\bigcup_{n \in \omega} \{n \} \times S(n)$, for a slalom $S$. It is easily checked that this is an analytic P-ideal. Again let $p$ consist of the graphs of functions. Then $\mathbb{P}_{\not\ni^*}$ corresponds to $\mathbb{P}_K \restriction p$.
\begin{thm}
   $\not\ni^*$ has the total universality property and $\mathbb{P}_{\not\ni^*}$ is proper, etc.
\end{thm}

$\mathbb{P}_{\not\ni^*}$ provides a natural way to increase $\non(\mathcal{M})$ (see \cite[Lemma 2.4.7]{BartoszynskiJudah1995} for the fact that $\non(\mathcal{M})= \mathfrak{b}(\not\ni^*)$). ``True'' eventually different forcing $\mathbb{P}_{\neq^*}$ would be another option, but, as we shall see in the next section, it is not proper. In \cite{BartoszynskiJudah1995}, the proper forcing notion $\mathbb{PT}_{f,g}$ is presented to optimally increase $\non(\mathcal{M})$. It would be interesting to know how it relates to $\mathbb{P}_{\not\ni^*}$.

\begin{quest}
    Is there a tree dichotomy for the analytic $\not\ni^*$-dominating families?
\end{quest}

\subsubsection{Slalom based localization forcing}

\begin{definition}
    Let $S,S'$ be slaloms as above. Then we write $S' \subseteq_{\operatorname{sl}}^* S$ to say that $\forall^\infty n (S'(n) \subseteq S(n))$.
\end{definition}

Consider $K$ as before and let $p$ consist of the complements of elements of $K$. Then $\mathbb{P}_{\subseteq_{\operatorname{sl}}^*}$ is equivalent to $\mathbb{P}_K \restriction p$ and adds a $\subseteq_{\operatorname{sl}}^*$-dominating slalom. This is a natural way to increase $\add(\mathcal{N})$, see \cite[3.2.A]{BartoszynskiJudah1995}. It is conceivable that ``true localisation forcing'', i.e. $\mathbb{P}_{\in^*}$, where $x \in^* S \leftrightarrow \forall^\infty n(x(n) \in S(n))$, is not proper either by a similar argument as for eventual difference, but we haven't studied this. 

\begin{thm}\label{thm:domforce}
   $\subseteq_{\operatorname{sl}}^*$ has the total universality property and $\mathbb{P}_{\subseteq_{\operatorname{sl}}^*}$ is proper, etc.
\end{thm}

\begin{quest}
    Is there a tree dichotomy for the analytic $\subseteq_{\operatorname{sl}}^*$-dominating families?
\end{quest}

\section{Two improper forcing notions}\label{sec:nonexample}

\subsection{True eventually different forcing}

\begin{definition}
    For $x,y \in \omega^\omega$, we write $x \neq^* y$ if $\forall^\infty n (x(n) \neq y(n))$ and say that $x$ and $y$ are eventually different. 
\end{definition}

\begin{thm}
    $\neq^*$ does not have the universality property. In fact, for any countable $M$, there is an eventually different real $d$ over $M$, so that for any real $c \notin M$, $d$ is not eventually different over $M[c]$. 
\end{thm}

\begin{proof}
  Let $M$ be a countable transitive model.  Working in $M$, find for each $s \in \omega^{<\omega}$, a set $a_s \in [\omega]^\omega$, so that the following conditions are satisfied:

    \begin{enumerate}
        \item For all $s \subsetneq t \in \omega^{<\omega}$, $a_s \cap a_t = \emptyset$,
        \item for any finite antichain $S \subseteq \omega^{<\omega}$, $\bigcap_{s \in S} a_s$ is infinite,
        \item for any $x \in \omega^\omega$, $\bigcup_{n \in \omega} a_{x \restriction n} = \omega$.
    \end{enumerate}

    This is fairly straightforward: Let $\langle S_n : n \in \omega \rangle$ enumerate all finite antichains of $\omega^{<\omega}$, say each one appearing infinitely often. Simply let $$a_s := \begin{cases}
        \{ n \geq \vert s\vert : s \in S_n \} \cup \{ \vert s \vert \} &\text{ if } \forall t \subseteq s (t \notin S_{\vert s\vert}) \\
        \{ n \geq \vert s\vert : s \in S_n \} &\text{ otherwise.}
    \end{cases}$$
    It is fairly easy to check that (1), (2) and (3) are satisfied.
    
    Next, let $k \colon \omega^{<\omega} \to \omega$ be some injection. For each $s \in \omega^{<\omega}$, we define $p_s$ to be the function with domain $a_s$ and constant value $k(s)$. For each $x \in \omega^\omega$, let $f(x) := \bigcup_{n \in \omega} p_{x \restriction n}$. Then $f\colon \omega^\omega \to \omega^\omega$ is a continuous function coded in $M$. 

    \begin{claim}\label{claim:evdiff1}
        Let $x_0, \dots, x_{n-1} \in \omega^\omega$ and $S \subseteq \omega^{<\omega}$ be an antichain. Then for all but $n$-many $s \in S$, there are infinitely many $m \in a_s$, so that $p_s(m) \neq x_i(m)$, for all $i < n$.
    \end{claim}

    \begin{proof}
        Otherwise, there are $n+1$-many $s_0, \dots, s_n \in S$, so that for each $j \leq n$, there is $k_j \in \omega$, where for all $m \in a_{s_j} \setminus k_j$. $p_{s_j}(m) = x_i(m)$, for some $i < n$. Let $m \in \bigcap_{j \leq n} {a_{s_j}}$, $m > \max_{j \leq n} k_j$, be arbitrary. Then the values $p_{s_{j}}(m)$, $j \leq n$, are pairwise distinct; so there are $n+1$-many such values. But there at most $n$-many values $x_i(m)$, $i < n$. So for some $j \leq n$, $p_{s_j}(m)$ is different to all those -- contradiction.
    \end{proof}

    \begin{claim}\label{claim:evdiff2}
         Let $x_0, \dots, x_{n-1} \in \omega^\omega$. Then there are $y_0, \dots, y_{n-1} \in \omega^{\omega}$ so that for any $s \in \omega^{<\omega}$ which is not an initial segment of any $y_i$, $i < n$, there are infinitely many $m \in a_s$, so that $p_s(m) \neq x_i(m)$, for all $i < n$.
    \end{claim}

    \begin{proof}
    The set $B = \{s \in \omega^{<\omega} : \forall^\infty m \in a_s \exists i < n (p_s(m) = x_i(m))\}$ contains antichains of size at most $n$, according to the previous claim. It follows from Dilworth's theorem, that $B$ can be covered by $n$-many chains (see \cite{Dilworth}).\footnote{Dilworth's theorem is usually stated for finite posets: If $P$ is finite and has no antichain of size $n+1$, then $P$ can be covered by $n$-many chains. A simple compactness argument shows of course that the same holds for infinite posets and fixed $n \in \omega$.}
    \end{proof}

\begin{definition}\label{def:Sy}
    Let $\bar y = \langle y_i : i < n \rangle \subseteq \omega^\omega$. Then we define $S_{\bar y} \subseteq \omega^{<\omega}$ to be the set of $\subseteq$-minimal $s \in \omega^{<\omega}$ of length $\geq n$ not extended by any of the $y_i$, $i < n$.
\end{definition}

$S_{\bar y}$ is a \emph{front} for $\omega^\omega \setminus \{ y_0, \dots, y_{n-1} \}$, that is, $S_{\bar y}$ is an antichain and $ \bigcup_{s \in S_{\bar y}} [s] = \omega^\omega \setminus \{ y_0, \dots, y_{n-1} \}$. For any $\bar x = \langle x_0, \dots, x_{n-1}\rangle \in M$, let $\bar y_{\bar x} = \langle y_0, \dots, y_{n-1} \rangle \in M$ be as in the previous claim. For simplicity, we write $A_{\bar x} = S_{\bar y_{\bar x}}$.
    
   We will now construct an eventually different real $d$ over $M$ so that for any real $c\notin M$, $f(c) =^\infty d$. Enumerate $M \cap \omega^\omega$ as $\bar x = \langle x_n : n \in \omega \rangle$. Enumerate each $A_{\bar x \restriction n}$ as $\langle s_{n, i} : i \in \omega \rangle$. Further, let $\langle (n_j, i_j) : j \in \omega \rangle$ enumerate all pairs $(n,i) \in \omega^2$ and let $t_j := s_{n_j, i_j}$. 
   
   For every $j \in \omega$, there is a largest number $l_j \leq j$, so that for infinitely many $m \in a_{t_j}$, we have $p_{t_j}(m) = k(t_j) \neq x_l(m)$, for all $l < l_j$. Note that the sequence $\langle l_j : j \in \omega \rangle$ must diverge to infinity. Namely, fixing any $l \in \omega$, we show that $l_j \geq l$ eventually. First, $l_j \geq l$ whenever $t_j \in A_{\bar x \restriction n}$ for some $n \geq l$. Secondly, according to Claim~\ref{claim:evdiff1}, for any fixed $A_{\bar x \restriction n}$, there can be at most finitely many $t_j \in A_{\bar x \restriction n}$, for which $l_j < l$.
   
   Now define an increasing sequence $\langle m_j : j \in \omega \rangle$, so that for each $j$, $m_j \in a_{t_j}$ and $p_{t_j}(m_j)$ is distinct from all $x_l(m_j)$, $l < l_j$. Letting $d_0 = \{ (m_j, p_{t_j}(m_j)) : j \in \omega \}$, we then have that $d_0$ is eventually different over $M$ and we may extend $d_0$, if necessary, to a total eventually different $d \in \omega^\omega$ over $M$.

   This works. If $c \in \omega^\omega \setminus M$ is arbitrary, then for each $n \in \omega$, there is $s \in A_{\bar x \restriction n}$, $s \subseteq c$. Then $s = t_j$ for some $j$, and $f(c)(m_j) = p_{t_j}(m_j) = d(m_j)$. Picking larger and larger $n$, we find infinitely many $m$, with $f(c)(m) = d(m)$. Thus $d$ is not eventually different over $M[c]$.
   \end{proof}

We will now turn this argument into a non-properness proof.

\begin{thm}\label{thm:edcoll}
    $\mathbb{P}_{\neq^*}$ collapses the continuum to $\omega$ below some condition.
\end{thm}

\begin{proof}
    Let $\langle a_s, p_s : s \in \omega^{<\omega} \rangle$ and $f$ be as before. We consider the set $p \subseteq \omega^\omega$ consisting of those $d \in \omega^\omega$, so that for each $n \in \omega$, there are $y^n_0, \dots, y^n_{n-1} \in \omega^\omega$, such that for every $t \in S_{\bar y^n}$, there is $s \subseteq t$, $\vert s \vert \geq n$, with $d(m) = p_s(m)$ for some $m \in a_s$.
    
    Given any $\langle x_n : n \in \omega \rangle \subseteq \omega^\omega$ let $M$ be a countable transitive model containing these reals. The eventually different real $d$ over $M$ produced in the previous proof is an element of $p$. This shows that $p$ is an $\omega$-eventually different family. 
    
    $p$ is also Borel, in fact $G_{\delta}$. Namely, check that $d \in p$ if and only if for each $n \in \omega$ and each antichain $S \subseteq \omega^{\geq n}$ of size $n+1$, there is some $t \in S$ and $s \subseteq t$, $\vert s \vert \geq n$, such that $p_s(m) = d(m)$, for some $m \in a_s$. The direction from left to right is immediate. For the other direction, given $n \in \omega$, let $$B = \{ t \in \omega^{\geq n} : \neg \exists s \subseteq t, m \in a_s (\vert s \vert \geq n \wedge p_s(m) = d(m))\}$$ and apply Dilworth's theorem again.

Now suppose $d \in p$ is $\mathbb{P}_{\neq^*}$ generic over $V$. Working in $V[d]$, for each $n \in \omega$, we find $y^n_0, \dots, y^n_{n-1} \in \omega^\omega$ witnessing that $d \in p$. Simply note that $V \cap \omega^\omega \subseteq \{ y^n_i : n \in \omega, i \leq n \}$ must be the case. Otherwise, if $y \in V \cap \omega^\omega \setminus \{ y^n_i : n \in \omega, i \leq n \}$, then for each $n \in \omega$, there is $s \subseteq y$, $\vert s \vert \geq n$, and some $m \in a_s$, with $p_s(m) = d(m)$. Followingly $f(y)(m) = d(m)$, for infinitely many $m$. But $f(y) \in V$ and $d$ is eventually different over $V$ -- contradiction.\end{proof}

\begin{quest}
    Is true eventually different forcing proper somewhere? Is there some condition $B$, such that each eventually different real $d \in B$ over a model $M \ni B$ is universally eventually different over $M$?
\end{quest}

\subsection{Refining forcing}

\begin{definition}
    For $x, y \in [\omega]^\omega$, let us say that $y$ \emph{refines} $x$ if either $y \subseteq^* x$ or $y \subseteq^* \omega \setminus x$. 
\end{definition}

This is more often called \emph{unsplitting} or \emph{reaping}, but ``refining" sounds neither too artificial nor has an aggressive connotation to it, which makes it for a good middle ground. The term is borrowed from Spinas, see \cite{Spinas2008}. Adding a refining real increases the splitting number $\mathfrak{s}$.

\begin{thm}\label{thm:reapnotuniv}
    Refining does not have the universality property. In fact, for any countable transitive $M$, there is a reaping real $r$ over $M$, so that for any $c \notin M$, $r$ is not refining over $M[c]$. 
\end{thm}

\begin{definition}
    Let $P$ be a set of partial functions $\omega \to 2$. Then we say that $P$ is an \emph{independent family} if for any $p_0, \dots, p_{n-1} \in P$ and any $h \colon n \to 2$, $$\bigcap_{i < n} p_i^{-1}(h(i))$$ is infinite.
\end{definition}

\begin{lemma}\label{lem:partialfunctions}
    There are partial functions $\langle p_s : s \in \omega^{<\omega} \rangle$ from $\omega$ to $2$ so that the following are satisfied:

\begin{enumerate}
    \item For all $s \subsetneq t \in \omega^{<\omega}$, $\dom p_s \cap \dom p_t = \emptyset$. 
    \item For any finite antichain $S \subseteq \omega^{<\omega}$, $\{ p_s : s \in S \}$ is an independent family. 
\end{enumerate}
\end{lemma}

\begin{proof}
We consider $Q^{<\omega}$ instead of $\omega^{<\omega}$, where $Q$ is the set of rational numbers. Let $\langle S_n : n \in \omega \rangle$ enumerate all finite antichains of $Q^{<\omega}$. For $s \in Q^{<\omega}$, we define $p_s(n) = j$ if $s \in S_n$ and in the lexicographical order on $S_n$ induced from the linear order on $Q$, $s$ is at a position with parity $j$. Clearly (1) is satisfied as $s \subsetneq t$ implies that $s$ and $t$ are in no common antichain. Now if $S \subseteq Q^{<\omega}$ is a finite antichain and $h \colon S \to 2$ is arbitrary, we find infinitely many finite antichains $S_n$, extending $S$, in which each $s \in S$ is in a position with parity $h(s)$ -- in other words $p_s(n) = h(s)$, for each $s \in S$. This shows that $\bigcap_{s \in S} p_s^{-1}(h(s))$ is infinite, as required.
\end{proof}

\begin{lemma}[Combinatorial Antichain Lemma]
    For any $n \in \omega$, whenever $A_i \subseteq \omega^{<\omega}$, $i < n$, are antichains of size at least $n$, there are distinct $s_0 \in A_0, \dots, s_{n-1} \in A_{n-1}$, so that $\{s_i : i < n \}$ forms an antichain.
\end{lemma}

\begin{proof}
    The statement is vacuous for $n = 0$ and trivial for $n = 1$. So assume now that the statement is true for some $n \geq 1$. Let $A_i \subseteq \omega^{<\omega}$, $i \leq n$, be antichains, each of size $n+1$.
    
    Once we find any $i \leq n$ and $s \in A_i$, where $s$ is incompatible with $n$-many $t \in A_j$, for each $j \in (n+1) \setminus \{ i \}$, we are done: Simply find distinct $s_j \in A_j' = \{ t \in A_j  : s \perp t \}$, where $\{s_j: j \in (n+1) \setminus \{ i \} \}$ forms an antichain. Then so does $\{s_j : j \in (n+1) \setminus \{ i \} \} \cup \{ s \}$. 
    
    So suppose this is not possible. Let $s_0 \in \bigcup_{i \leq n} A_i$ be completely arbitrary. Then there is some $j \leq n$ with at least two nodes $s_1, t_1 \in A_j$, both compatible with $s_0$. The only way is if $s_0$ is a strict initial segment of both $s_1$ and $t_1$. Otherwise $s_1$ and $t_1$ are compatible, but $A_j$ is an antichain. In precisely the same way, we find a strict extension $s_2 \in \bigcup_{i \leq n} A_i$ of $s_1$, and so on. This process never ends, but $\bigcup_{i \leq n} A_i$ is finite -- contradiction.\end{proof} 

\begin{lemma}
 Let $x_0, \dots, x_{n-1} \in [\omega]^\omega$. Then there is $\eta \in {}^n2$ so that for any antichain $S \subseteq \omega^{<\omega}$, for all but $(2^n-1)$-many $s \in S$,
 \begin{itemize}
        \item $p_s^{-1}(0) \cap \bigcap_{i < n} x_i^{\eta(i)}$ is infinite, 
        \item and $p_s^{-1}(1) \cap \bigcap_{i < n} x_i^{\eta(i)}$ is infinite.
 \end{itemize}
\end{lemma}

\begin{proof}
    Suppose towards a contradiction that for each $\eta \in {}^n2$, there is an antichain $A_\eta \subseteq \omega^{<\omega}$ of size $2^n$, so that for every $s \in A_\eta$, the above fails. Pick $s_\eta \in A_\eta$, $\eta \in {}^n2$, distinct so that $\{s_\eta : \eta \in {}^n2 \}$ forms an antichain. 

    For each $\eta \in {}^n2$, letting $z_\eta = \bigcap_{i < n} x_i^{\eta(i)}$, we thus have that $$z_\eta \subseteq^* \omega \setminus p_{s_\eta}^{-1}(0) \text{ or } z_\eta \subseteq^* \omega \setminus p_{s_\eta}^{-1}(1).$$ Let $h(\eta) \in \{0,1\}$ be least so that $z_\eta \subseteq^* \omega \setminus p_{s_\eta}^{-1}(h(\eta))$. Then we have that $$\omega = \bigcup_{\eta \in {}^n2} z_\eta \subseteq^* \bigcup_{\eta \in {}^n2} \omega \setminus p^{-1}(h(\eta)) =^* \omega.$$

    So $$ \emptyset =^* \omega \setminus \left(\bigcup_{\eta \in {}^n2} \omega \setminus p^{-1}(h(\eta))\right) = \bigcap_{\eta \in {}^n2} p^{-1}(h(\eta)).$$
    
    But the independence of $\{p_{s_\eta} : \eta \in {}^n2\}$ implies that $\bigcap_{\eta \in {}^n2} p^{-1}(h(\eta))$ is infinite.
\end{proof}

\begin{lemma}
 Let $x_0, \dots, x_{n-1} \in [\omega]^\omega$. Then there is $\eta \in {}^n2$ and there are $y_i \in \omega^\omega$, $i < 2^n-1$, so that for any $s \in \omega^{<\omega}$ that is not an initial segment of any of the $y_i$, 
 \begin{itemize}
        \item $p_s^{-1}(0) \cap \bigcap_{i < n} x_i^{\eta(i)}$ is infinite, 
        \item and $p_s^{-1}(1) \cap \bigcap_{i < n} x_i^{\eta(i)}$ is infinite.
 \end{itemize}
\end{lemma}

\begin{proof}
    Using Dilworth's theorem, exactly as before.
\end{proof}

\begin{proof}[Proof of Theorem~\ref{thm:reapnotuniv}]
    For each $x \in \omega^\omega$, let $p(x) = \bigcup_{n \in \omega} p_{x \restriction n}$ and $f(x) = p(x) \cup \{ (n, 0) : n \notin \dom p(x) \}$. We construct a reaping real $r$ over $M$, so that for any $c \in \omega^\omega \setminus M$, $f(c)$ splits $r$ (viewed as a subset of $\omega$). 

    Enumerate $M \cap \omega^\omega$ as $\bar x = \langle x_n : n \in \omega \rangle$. For each $n$, there is $\eta_n \in {}^n2$ and a finite sequence $\bar y^n = \langle y^n_i : i < 2^n - 1 \rangle \in M$, in accordance to the previous lemma. By compactness -- possibly outside of $M$ -- we find a single $\eta \colon \omega \to 2$, so that for every $n$, there is $n' \geq n$ with $\eta_{n'} \restriction n \subseteq \eta$. Define $A_n := S_{\bar y^{n'}}$ (see Definition~\ref{def:Sy}). Then we have that for each $n \in \omega$ and for each $s \in A_n$, 

 \begin{itemize}
        \item $p_s^{-1}(0) \cap \bigcap_{i < n} x_i^{\eta(i)}$ is infinite, 
        \item and $p_s^{-1}(1) \cap \bigcap_{i < n} x_i^{\eta(i)}$ is infinite.
 \end{itemize}

The rest is as in Theorem~\dots. Each $A_n$ is enumerated as $\langle s_{n,i} : i \in \omega \rangle$. Let $\langle (n_j, i_j) : j \in \omega \rangle$ enumerate all pairs $(n, i) \in \omega^2$ and let $t_j := s_{n_j, i_j}$. For each $j$, there is a maximal $l_j \leq j$, so that 

 \begin{itemize}
        \item $p_{t_j}^{-1}(0)\cap\bigcap_{i < l_j} x_i^{\eta(i)}$ is infinite, 
        \item and $p_{t_j}^{-1}(1)\cap\bigcap_{i < l_j} x_i^{\eta(i)}$ is infinite.
 \end{itemize}

 Again, we find that $\langle l_j : j \in \omega \rangle$ diverges to infinity: $l_j \geq l$ whenever $n_j \geq l$, and there are at most finitely many $j$ where $n_j < l$ and $\bigcap_{i < l} x_i^{\eta(i)}$ is almost contained in $p_{t_j}^{-1}(0)$ or $p_{t_j}^{-1}(1)$.

Define an increasing sequence $\langle m_{2j}, m_{2j +1} : j \in \omega \rangle$, where $$m_{2j} \in p_{t_j}^{-1}(0) \cap \bigcap_{i < l_j} x_i^{\eta(i)} \text{ and } m_{2j+1} \in p_{t_j}^{-1}(1) \cap \bigcap_{i < l_j} x_i^{\eta(i)}.$$

We let $r = \{ m_j : j \in \omega \}$. Then $r$ is a reaping real over $M$. On the other hand, whenever $c \in \omega^\omega \setminus M$, then for each $n\in \omega$, there is $s \in A_n$ with $s \subseteq c$. For some $j$, $t_j = s$, so $f(c)(m_{2j}) = p_s(m_{2j}) = 0$ and $f(c)(m_{2j +1}) =p_s(m_{2j +1}) = 1$. Consequently, $f(c)$ splits $r$.\end{proof}

\begin{thm}\label{thm:refcoll}
    Refining forcing, i.e. the forcing consiting of Borel $\omega$-refining families, collapses the continuum to $\omega$ below some condition.
\end{thm}

\begin{proof}
    The condition $p$ consists of all $r \in [\omega]^\omega$, where for each $n \in \omega$, there is a finite sequence $\bar y \subseteq \omega^\omega$ of length at least $n$, so that $p_s^{-1}(0) \cap r \neq \emptyset$ and $p_s^{-1}(1) \cap r\neq \emptyset$, for each $s \in S_{\bar y}$. This is a Borel set (see the argument in Theorem~\ref{thm:edcoll}) and the previous proof shows that it is an $\omega$-reaping family. 
    
    Whenever $r \in p$ is reaping over $V$, let $\bar y^n = \langle y^n_i : i < N_n \rangle$ be as stipulated above, for each $n$. We claim that $\omega^\omega \cap V \subseteq \{ y^n_{i} : n \in \omega, i < N_n \}$, so $\omega^\omega \cap V$ is countable in any model extending $V$ and containing $r$ : Let $y \in \omega^\omega \cap V$ be arbitrary. Then $r$ reaps $f(y) \in V$, i.e. $r \subseteq^* f(y)^{-1}(0)$ or $r \subseteq^* f(y)^{-1}(1)$. In particular, there is some $n \in \omega$ so that for all $n' \geq n$, $r \cap p_{y \restriction n'}^{-1}(1) = \emptyset$ or $r \cap p_{y \restriction n'}^{-1}(0) = \emptyset$. It follows that $y \in \{ y^n_i : i < N_n \}$, since otherwise, for $n' \geq n$ least with $y \restriction n'$ not contained in any $y^n_i$, $y \restriction n' \in S_{\bar y^n}$, contradicting the choice of $\bar y^n$.
\end{proof}

There are a few more things of interest that we can say or ask about refining forcing. It seems to be a difficult problem to characterise $\omega$-refining Borel sets and we do not know if the ideal is $\mathbf\Delta^1_2$ on $\mathbf\Sigma^1_1$. Since refining forcing is not proper, it is in particular not the same as Mathias forcing, i.e. forcing with Ellentuck-open sets, as one might maybe naively suspect.\footnote{On this matter, the ideal associated to Mathias forcing is not $\mathbf\Delta^1_2$ on $\mathbf{\Sigma}^1_1$ by \cite{Sabok2012}.} It is not hard to come up with a counterexample either. First recall that Ellentuck-open sets, or ``Ellentuck cubes", are of the form $$[s, a]^\omega := \{ x \in [\omega]^\omega : x \cap \max(s) +1 = s \text{ and } x \setminus s \subseteq a \},$$ for $s \in [\omega]^{<\omega}$ and $a \in [\omega]^\omega$. 

Consider $$B = \{ x \in [\omega]^\omega : \forall n \in \omega ( x \cap [2^n, 2^{n+1}] \neq \emptyset \rightarrow \vert x \cap [2^n, 2^{n+1}] \vert = 2) \}.$$ This is a closed $\omega$-refining family. We leave the details to the reader. But clearly $B$ cannot contain an Ellentuck-cube. Something that comes close to this can be said nevertheless:

\begin{thm}
    Let $B$ be a Borel refining family. Then there is a finite-to-one function $f \colon \omega \to \omega$ such that $[s,a]^\omega \subseteq \{ f[x] : x \in B \}$ for some $s$ and $a$.
\end{thm}

Note that one can easily require $s = \{0\}$ and $a = \omega \setminus \{0\}$ by modifying $f$. 

\begin{proof}
    Consider the sets $$A_{0,1} = \{ y \in [\omega]^\omega : \exists u,v \in B (u \subseteq^* y \wedge v \subseteq^* \omega \setminus y) \}$$ and $$A_0 = \{ y \in [\omega]^\omega : y \notin A_{0,1} \wedge \exists u \in B (u \subseteq^* y) \}.$$

    \begin{claim}\label{claim:bndfilter}
        $A_0$ is contained in a $\sigma$-compact subset of $[\omega]^\omega$. In particular, there is a strictly increasing sequence $\langle i_n : n \in \omega \rangle$ such that for all $y \in A_0$, $y \cap [i_n, i_{n+1}) \neq \emptyset$ for all but finitely many $n$.
    \end{claim}
    \begin{proof}
        We note that for any $y_0, y_1 \in A_0$, $y_0 \cap y_1$ is infinite. Namely, let $x \in B$ refine $y_0$. By definition of $A_0$, the only possibility is if $x \subseteq^* y_0$. Since $y_1 \in A_0$, it cannot be that $x \subseteq^* \omega \setminus y_1$, so $x$ intersects $y_1$ infinitely and so does $y_0$. Now if $A_0$ is not contained in a $\sigma$-compact set then it contains a superperfect set, see e.g. \cite[21.24]{Kechris1995} for more details. But it is easy to find two reals with finite intersection in any superperfect set.
    \end{proof}

 Let $\langle i_n : n \in \omega \rangle$ be as claimed, with $i_0 = 0$, and let $f : \omega \to \omega$ be the finite-to-one function mapping each interval $[i_n,i_{n+1})$ to $n$. Now if $y \in [\omega]^\omega$ is co-infinite, there is $x \in B$, with $x \subseteq^*  \bigcup_{n \in y} [i_n,i_{n+1})$, because $\bigcup_{n \notin y} [i_n,i_{n+1}) \notin A_0$ and $B$ is reaping. Thus there is $x \in B$ with $f[x] \subseteq^* y$. As $y$ was arbitrary, this means that $$X = \{f[x] : x \in B \}$$ is dense in $([\omega]^\omega, \subseteq^*)$. It follows from the Galvin-Prikry Theorem, see \cite[19.C]{Kechris1995}, that $[s, a]^{\omega} \subseteq X$ for some $s$ and $a$.
\end{proof}

\begin{lemma}
    Let $B$ be an $\omega$-refining family and $f \in \omega^\omega$. For any $a \subseteq \omega$, let $B_{f,a} = \{ x \in B : f[x] \subseteq^* a \}$. Then $\mathcal{I}_{f}(B) = \{ a : B_{f,a} \text{ is not $\omega$-reaping} \}$ is an ideal on $\omega$.
\end{lemma}

\begin{proof}
    Clearly $\mathcal{I}_f(B)$ is closed under subsets and does not contain $\omega$. Suppose we have $a_0\cup \dots \cup a_n = b$, $B_{f,b}$ is $\omega$-refining, but $B_{f,a_i}$ is not, for all $i \leq n$. Let $C_i \subseteq [\omega]^\omega$ countable witness that $B_{f, a_i}$ is not $\omega$-refining. On the other hand, let $x \in B_{f,b}$ refine all sets $f^{-1}(a_i)$, $i \leq n$, and all elements of $\bigcup_{i \leq n} C_i$ simultaneously.
    
    Note that $f$ must be finite-to-one on $x$. Otherwise, there is an infinite $x' \subseteq x$ with $f$ constant on $x'$. But then for example $x' \in B_{f,a_0}$ and $x'$ also refines $C_0$.
    
    We obtain that $x \subseteq^* f^{-1}(b) = f^{-1}(a_0) \cup \dots \cup f^{-1}(a_n)$. As $x$ refines each set $f^{-1}(a_i)$, there must be at least one $i \leq n$ such that $x \subseteq^* f^{-1}(a_i)$. But then again $x \in B_{f,a_i}$, and $x$ refines $C_i$ -- contradiction.
\end{proof}

Note of course that if $f$ is finite-to-one, then $\mathcal{I}_{f}(B)$ contains all finite sets. Also, if $A \leq B$, then $\mathcal{I}_f(B) \subseteq \mathcal{I}_{f}(A)$. In the following, we say that a subset of $[\omega]^\omega$ is $K_\sigma$-regular if it is either contained in a $\sigma$-compact set or contains a superperfect set.

\begin{thm}
   Assume that every $\mathbf\Sigma^1_2$ set is either Baire measurable or $K_\sigma$-regular. Then refining forcing collapses the continuum to $\mathfrak{d}$.
\end{thm}

\begin{proof}
  Let $\{g_\alpha : \alpha < \mathfrak{d} \}$ be a dominating family consisting of strictly increasing functions. For each $\alpha$, let $f_{\alpha}$ be the finite-to-one function mapping an interval $[g(n), g(n+1))$ to $n$. Now suppose that $B$ is an arbitrary condition. The ideal $\mathcal{I}(B) = \mathcal{I}_{\operatorname{id}}(B)$ is $\Sigma^1_2$, so, by the regularity assumptions, there is $\alpha < \mathfrak{d}$, such that for any $x \in \mathcal{I}(B)^*$, for all but finitely many $n$, $x \cap [g_{\alpha}(n), g_{\alpha}(n+1)) \neq \emptyset$.   In the case of Baire measurability, this is due to a well-known result of Talagrand, see \cite[Theorem 4.1.2]{BartoszynskiJudah1995}. For $K_\sigma$-regularity, this is the same as Claim~\ref{claim:bndfilter}.
  
  In particular, for any infinite $a \subseteq \omega$, $$\bigcup_{n \in a} [g_{\alpha}(n), g_{\alpha}(n+1)) = f_\alpha^{-1}(a) \in \mathcal{I}(B)^+.$$
  In other words, $B_{f_\alpha,a}\subseteq B$ is $\omega$-refining and $\mathcal{I}_{f_\alpha}(B)$ is the ideal of finite sets. 
  
  Recall that for any real $x \in 2^\omega$, there is a set $a \in [\omega]^{\omega}$ such that $x$ is computable from any infinite subset of $a$.\footnote{For instance, pick an effective bijection $e\colon 2^{<\omega} \to \omega$ and let $a = e[\{x \restriction n : n \in \omega \}]$.} Thus, since for any $r \in B_{f_\alpha,a}$, $f_\alpha[r] \subseteq^* a$, $$B_{f,a} \Vdash \check x \text{ is computable from } f_\alpha[\dot r].$$

By genericity, we obtain that every ground model real is computable from $f_\alpha[r]$, for some $\alpha < \mathfrak{d}$ and $r$ the generic refining real.
\end{proof}

\begin{quest}
    Is refining forcing proper somewhere? Is there some Borel $\omega$-refining family $B$ such that each refining real $r \in B$ over a model $M \ni B$ is universally refining over $M$?
\end{quest}

\bibliographystyle{plain}
\bibliography{ref}

@article {Solecki1994,
    AUTHOR = {Sławomir Solecki},
     TITLE = {Covering analytic sets by families of closed sets},
   JOURNAL = {J. Symbolic Logic},
  FJOURNAL = {The Journal of Symbolic Logic},
    VOLUME = {59},
      YEAR = {1994},
    NUMBER = {3},
     PAGES = {1022--1031},
      ISSN = {0022-4812,1943-5886},
   MRCLASS = {54H05 (03E15 04A15)},
  MRNUMBER = {1295987},
MRREVIEWER = {Edward\ Azoff},
       DOI = {10.2307/2275926},
       URL = {https://doi.org/10.2307/2275926},
}

@article {Kellner2012,
    AUTHOR = {Kellner, J.},
     TITLE = {Non elementary proper forcing},
   JOURNAL = {Rend. Semin. Mat. Univ. Politec. Torino},
  FJOURNAL = {Rendiconti del Seminario Matematico. Universit\`a{} e
              Politecnico Torino},
    VOLUME = {70},
      YEAR = {2012},
    NUMBER = {3},
     PAGES = {207--246},
      ISSN = {0373-1243,2704-999X},
   MRCLASS = {03E35 (03E40)},
  MRNUMBER = {3305544},
MRREVIEWER = {Andrzej\ Ros\l anowski},
}

@article {Shelah2004,
    AUTHOR = {Shelah, S.},
     TITLE = {Properness without elementaricity},
   JOURNAL = {J. Appl. Anal.},
  FJOURNAL = {Journal of Applied Analysis},
    VOLUME = {10},
      YEAR = {2004},
    NUMBER = {2},
     PAGES = {169--289},
      ISSN = {1425-6908,1869-6082},
   MRCLASS = {03E40 (03E05 03E47)},
  MRNUMBER = {2115943},
MRREVIEWER = {Andrzej\ Ros\l anowski},
       DOI = {10.1515/JAA.2004.169},
       URL = {https://doi.org/10.1515/JAA.2004.169},
}

@article {AsperoKaragila,
    AUTHOR = {Asper\'o, David and Karagila, Asaf},
     TITLE = {Dependent choice, properness, and generic absoluteness},
   JOURNAL = {Rev. Symb. Log.},
  FJOURNAL = {The Review of Symbolic Logic},
    VOLUME = {14},
      YEAR = {2021},
    NUMBER = {1},
     PAGES = {225--249},
      ISSN = {1755-0203,1755-0211},
   MRCLASS = {03E25 (03A05 03E35 03E55 03E57)},
  MRNUMBER = {4277122},
MRREVIEWER = {Arnold\ W.\ Miller},
       DOI = {10.1017/S1755020320000143},
       URL = {https://doi.org/10.1017/S1755020320000143},
}

@incollection {Jackson2010,
    AUTHOR = {Jackson, Steve},
     TITLE = {Structural consequences of {AD}},
 BOOKTITLE = {Handbook of set theory. {V}ols. 1, 2, 3},
     PAGES = {1753--1876},
 PUBLISHER = {Springer, Dordrecht},
      YEAR = {2010},
      ISBN = {978-1-4020-4843-2},
   MRCLASS = {03E60 (03E15 03E45)},
  MRNUMBER = {2768700},
MRREVIEWER = {Daniel\ W.\ Cunningham},
       DOI = {10.1007/978-1-4020-5764-9\_22},
       URL = {https://doi.org/10.1007/978-1-4020-5764-9_22},
}

@incollection {Larson2025,
    AUTHOR = {Larson, Paul B.},
     TITLE = {An introduction to {${AD}^+$}},
 BOOKTITLE = {Higher recursion theory and set theory},
    SERIES = {Lect. Notes Ser. Inst. Math. Sci. Natl. Univ. Singap.},
    VOLUME = {44},
     PAGES = {145--182},
 PUBLISHER = {World Sci. Publ., Hackensack, NJ},
      YEAR = {[2025] \copyright 2025},
      ISBN = {978-981-98-0657-7; 978-981-98-0658-4; 978-981-98-0851-9},
   MRCLASS = {03E60},
  MRNUMBER = {4900331},
MRREVIEWER = {A.\ Kanamori},
       DOI = {10.1142/9789},
       URL = {https://doi.org/10.1142/9789},
}

@article{Solecki1999,
title = {Analytic ideals and their applications},
journal = {Annals of Pure and Applied Logic},
volume = {99},
number = {1},
pages = {51-72},
year = {1999},
issn = {0168-0072},
doi = {https://doi.org/10.1016/S0168-0072(98)00051-7},
url = {https://www.sciencedirect.com/science/article/pii/S0168007298000517},
author = {Sławomir Solecki}
}

@article{FarahZapletal2006,
title = {Four and more},
journal = {Annals of Pure and Applied Logic},
volume = {140},
number = {1},
pages = {3-39},
year = {2006},
note = {Cardinal Arithmetic at work: the 8th Midrasha Mathematicae Workshop},
issn = {0168-0072},
doi = {https://doi.org/10.1016/j.apal.2005.09.002},
url = {https://www.sciencedirect.com/science/article/pii/S0168007205001284},
author = {Ilijas Farah and Jindřich Zapletal}
}

@article{Dilworth,
 ISSN = {0003486X, 19398980},
 URL = {http://www.jstor.org/stable/1969503},
 author = {R. P. Dilworth},
 journal = {Annals of Mathematics},
 number = {1},
 pages = {161--166},
 publisher = {[Annals of Mathematics, Trustees of Princeton University on Behalf of the Annals of Mathematics, Mathematics Department, Princeton University]},
 title = {A Decomposition Theorem for Partially Ordered Sets},
 urldate = {2025-08-06},
 volume = {51},
 year = {1950}
}

@book{KanoveiSabokZapletal,
    AUTHOR = {Kanovei, Vladimir and Sabok, Marcin and Zapletal, Jind\v{r}ich},
     TITLE = {Canonical {R}amsey theory on {P}olish spaces},
    SERIES = {Cambridge Tracts in Mathematics},
    VOLUME = {202},
 PUBLISHER = {Cambridge University Press, Cambridge},
      YEAR = {2013},
     PAGES = {viii+269},
      ISBN = {978-1-107-02685-8},
   MRCLASS = {03-02 (03E05 03E15 03E40 05D10 91A44)},
  MRNUMBER = {3135065},
MRREVIEWER = {Miroslav\ Repick\'y},
       DOI = {10.1017/CBO9781139208666},
       URL = {https://doi.org/10.1017/CBO9781139208666},
}

@article{SchilhanCM,
    author = {Schilhan, Jonathan},
    title = {Canonical models for cardinal invariants of the continuum},
    journal = {in preparation},
    year = {202x}
}

@book {Moschovakis,
    AUTHOR = {Moschovakis, Yiannis N.},
     TITLE = {Descriptive set theory},
    SERIES = {Mathematical Surveys and Monographs},
    VOLUME = {155},
   EDITION = {Second},
 PUBLISHER = {American Mathematical Society, Providence, RI},
      YEAR = {2009},
     PAGES = {xiv+502},
      ISBN = {978-0-8218-4813-5},
   MRCLASS = {03-02 (03E15 54H05)},
  MRNUMBER = {2526093},
       DOI = {10.1090/surv/155},
       URL = {https://doi.org/10.1090/surv/155},
}

@article {BrendleHjorthSpinas,
    AUTHOR = {Brendle, J\"org and Hjorth, Greg and Spinas, Otmar},
     TITLE = {Regularity properties for dominating projective sets},
   JOURNAL = {Ann. Pure Appl. Logic},
  FJOURNAL = {Annals of Pure and Applied Logic},
    VOLUME = {72},
      YEAR = {1995},
    NUMBER = {3},
     PAGES = {291--307},
      ISSN = {0168-0072,1873-2461},
   MRCLASS = {03E15 (03E40 03E55 03E60)},
  MRNUMBER = {1327118},
MRREVIEWER = {Miroslav\ Repick\'y},
       DOI = {10.1016/0168-0072(94)00027-Z},
       URL = {https://doi.org/10.1016/0168-0072(94)00027-Z},
}

@article {Spinas2004,
    AUTHOR = {Spinas, Otmar},
     TITLE = {Analytic countably splitting families},
   JOURNAL = {J. Symbolic Logic},
  FJOURNAL = {The Journal of Symbolic Logic},
    VOLUME = {69},
      YEAR = {2004},
    NUMBER = {1},
     PAGES = {101--117},
      ISSN = {0022-4812,1943-5886},
   MRCLASS = {03E05 (03E15 03E17 54A35)},
  MRNUMBER = {2039350},
MRREVIEWER = {Boaz\ Tsaban},
       DOI = {10.2178/jsl/1080938830},
       URL = {https://doi.org/10.2178/jsl/1080938830},
}

@Book{BartoszynskiJudah1995,
  author =    {Bartoszynski, T. and Judah, H.},
  title =     {Set Theory: On the Structure of the Real Line},
  year =      {1995},
  series =    {Ak Peters Series},
  publisher = {Taylor \& Francis},
  isbn =      {9781568810447},
  url =       {https://books.google.at/books?id=Nkzbkj-c83YC},
  lccn =      {95034063}
}

@Book{Kechris1995,
  author =     {Kechris, Alexander S.},
  title =      {Classical descriptive set theory},
  year =       {1995},
  volume =     {156},
  series =     {Graduate Texts in Mathematics},
  publisher =  {Springer-Verlag, New York},
  isbn =       {0-387-94374-9},
  pages =      {xviii+402},
  doi =        {10.1007/978-1-4612-4190-4},
  url =        {http://dx.doi.org/10.1007/978-1-4612-4190-4},
  mrclass =    {03E15 (03-01 03-02 04A15 28A05 54H05 90D44)},
  mrnumber =   {1321597},
  mrreviewer = {Jakub Jasi\AA ski}
}

@article{Spinas2008,
author = {Otmar Spinas},
journal = {Fundamenta Mathematicae},
language = {eng},
number = {2},
pages = {179-195},
title = {Perfect set theorems},
url = {http://eudml.org/doc/283359},
volume = {201},
year = {2008},
}

@BOOK{Zapletal2004,
  title     = "Descriptive set theory and definable forcing",
  author    = "Zapletal, Jind{\v r}ich",
  publisher = "American Mathematical Society",
  month     =  sep,
  year      =  2014,
  address   = "Providence, RI",
  language  = "en",
url = "https://bookstore.ams.org/memo-167-793/"
}

@book {Zapletal2008,
    AUTHOR = {Zapletal, Jind\v{r}ich},
     TITLE = {Forcing idealized},
    SERIES = {Cambridge Tracts in Mathematics},
    VOLUME = {174},
 PUBLISHER = {Cambridge University Press, Cambridge},
      YEAR = {2008},
     PAGES = {vi+314},
      ISBN = {978-0-521-87426-7},
   MRCLASS = {03-02 (03E15 03E35 03E40 28A05 54A35)},
  MRNUMBER = {2391923},
MRREVIEWER = {Miroslav\ Repick\'{y}},
       DOI = {10.1017/CBO9780511542732},
       URL = {https://doi.org/10.1017/CBO9780511542732}}

@phdthesis{Khomskii,
    author = {Yurii Khomskii},
    title = {Regularity Properties and Definability in the Real Number Continuum},
    school = {Universiteit van Amsterdam},
    year = {2012}
}

@Inbook{Cummings2010,
author="Cummings, James",
title="Iterated Forcing and Elementary Embeddings",
bookTitle="Handbook of Set Theory",
year="2010",
publisher="Springer Netherlands",
address="Dordrecht",
pages="775--883",
isbn="978-1-4020-5764-9",
doi="10.1007/978-1-4020-5764-9_13",
url="https://doi.org/10.1007/978-1-4020-5764-9_13"
}

@InProceedings{FengMagidorWoodin,
author="Feng, Qi
and Magidor, Menachem
and Woodin, Hugh",
editor="Judah, Haim
and Just, Winfried
and Woodin, Hugh",
title="Universally Baire Sets of Reals",
booktitle="Set Theory of the Continuum",
year="1992",
publisher="Springer US",
address="New York, NY",
pages="203--242",
isbn="978-1-4613-9754-0"
}

@article{Sabok2012,
title = {Complexity of Ramsey null sets},
journal = {Advances in Mathematics},
volume = {230},
number = {3},
pages = {1184-1195},
year = {2012},
issn = {0001-8708},
doi = {https://doi.org/10.1016/j.aim.2012.03.001},
url = {https://www.sciencedirect.com/science/article/pii/S0001870812000941},
author = {Marcin Sabok}
}

\end{document}